%% file: arxiv-version.tex
\documentclass[12pt]{article}
\usepackage[margin=1in]{geometry}
\usepackage{lipsum}
\usepackage{bbm}
\usepackage[usenames]{color}
\usepackage{enumerate}
\usepackage{times}
\usepackage{textcomp}
\usepackage{cite}
\usepackage{url}
\usepackage{subcaption}
\usepackage{amsfonts,mathrsfs}
\usepackage{amssymb,amsmath}

\usepackage{amsthm}

\usepackage{verbatim}
\usepackage{acronym}
\usepackage{mathtools}

\usepackage{graphicx}
\usepackage{todonotes}

\usepackage{enumerate}
\usepackage{subfiles}
\usepackage[capitalize]{cleveref}

\newcommand{\remove}[1]{}

\def\Tr{\mathrm{tr}}
\def\E{\mathbb{E}}

\def \R{\mathbb{R}}

\def \Ecal {\mathcal{E}}
\def \Fcal {\mathcal{F}}
\def \Gcal {\mathcal{G}}
\def \Lcal {\mathcal{L}}
\def \Ncal {\mathcal{N}}
\def \Scal {\mathcal{S}}
\def \Vcal{[n]}

\def \e {E}

\def \diam {\mathrm{diam}({\mathcal{G}})}
\def \one {\mathbf{1}}

\def \rootl {\sqrt{\lambda}}

\def \emax {e_{\max}}

\def \Almg {A_{\textrm{lmg}}}

\def \Agossip {B}

\newtheorem{theorem}{Theorem}

\newtheorem{assumption}{Assumption}

\newtheorem{corollary}{Corollary}

\newtheorem{definition}{Definition}

\newtheorem{lemma}{Lemma}

\newtheorem{proposition}[theorem]{Proposition}
\newtheorem{remark}{Remark}

\usepackage{enumerate}

\renewcommand{\vec}[1]{\mathbf{#1}}
\newcommand{\sets}[1]{\mathcal{#1}}
\newcommand\inp[2]{\langle #1, #2 \rangle}
\usepackage[normalem]{ulem}

\def \vecw {\vec{w}}
\def \vecx {\vec{x}}
\def \bfx {\vec{x}}

\definecolor{OliveGreen}{rgb}{0,0.5,0}
\newcommand{\av}{\textcolor{black}}

\author{Ashwin Verma, Marcos M. Vasconcelos, Urbashi Mitra, and Behrouz Touri\footnote{
A.\ Verma and B.\ Touri are with the Department of Electrical and Computer Engineering, University of California San Diego 
     (email: a1verma@eng.ucsd.edu, btouri@eng.ucsd.edu), M.\ M.\ Vasconcelos is with the Commonwealth Cyber Initiative and the Department of Electrical and Computer Engineering, Virginia Tech (email: marcosv@vt.edu), and 
U. Mitra is with the Ming Hsieh Department of Electrical  Engineering, University of Southern California (email: ubli@usc.edu).  U. Mitra was supported in part by the following agencies: ONR under grant N00014- 15-1-2550, NSF under grants CNS-1213128, CCF-1718560, CCF-1410009, CPS-1446901 and AFOSR under grant FA9550-12-1-0215. M. M. Vasconcelos was supported by funding  from the Commonwealth Cyber Initiative (CCI). }}

\title{Maximal Dissent: a {State-Dependent} Way to Agree in Distributed Convex Optimization}
\date{}
\begin{document}

\maketitle

\begin{abstract}
Consider a set of agents collaboratively solving a distributed convex optimization problem, asynchronously, under stringent communication constraints. In such situations, when an agent is activated and is allowed to communicate with only one of its neighbors, we would like to pick the one holding the most informative local estimate. We propose new algorithms where the agents with maximal dissent average their estimates, leading to an information mixing mechanism that often displays faster convergence to an optimal solution compared to randomized gossip.
The core idea is that when two neighboring agents whose distance between local estimates is the largest among all neighboring agents in the network average their states, it leads to the largest possible immediate reduction of the quadratic Lyapunov function used to establish convergence to the set of optimal solutions.
As a broader contribution, we prove the convergence of max-dissent subgradient methods using a unified framework that can be used for other state-dependent distributed optimization algorithms. Our proof technique bypasses the need of establishing the information flow between any two agents within a time interval of uniform length by intelligently studying convergence properties of the Lyapunov function used in our analysis. 
\end{abstract}

\section{Introduction}

In distributed convex optimization, a collection of agents collaborate to minimize the sum of local objective functions by exchanging information over a communication network.
The primary goal is to design algorithms that converge to an optimal solution via local interactions dictated by the underlying communication network. 
A standard strategy to solving distributed optimization problems consists of each agent first combining the local estimates shared by its neighbors followed by a first-order subgradient method on its local objective function \cite{Tsitsiklis:1986,nedic2009distributed,nedic2010constrained}. Of particular relevance herein are the so-called {\em gossip} algorithms \cite{shah2009gossip}, where the information mixing step consists of averaging the sates of two agents connected by one of the edges selected from the network graph.

Two benefits of gossip algorithms are their simple asynchronous implementation and a reduction in communication costs. One common gossip algorithm is \textit{randomized}, in which an agent is randomly activated and chooses one of its neighbors randomly to average its state \cite{boyd2006randomized,ram2009asynchronous, Aysal:2009}. The randomization mechanism used in this gossip scheme is usually \textit{state-independent}\textbf{}. We consider a different approach to gossip in which the agent chooses one of its neighbors based on its state. At one extreme, we may think of agents who prefer to gossip with neighbors with similar ``opinions''. As in an \textit{echo-chamber}, where agents may only talk to others if they reinforce their own opinions, it does not lead to an effective information mixing mechanism. At the opposite extreme, we consider agents who prefer to gossip with neighbors with maximal disagreement or dissent. In this paper, we focus on the concept of \textit{max-dissent} gossip as a state-dependent information mixing mechanism for distributed optimization. We establish the convergence of the resulting distributed subgradient method under minimal assumptions on the underlying communication graph, and the local functions. 

The idea of enabling a consensus protocol to use state-dependent matrices dates back to the Hegselmann and Krause~\cite{hegselmann2002opinion} model for opinion dynamics. However, the literature on state-dependent averaging in distributed optimization is scarce and mostly motivated by applications in which the state represents the physical location of mobile agents (e.g. robots, autonomous vehicles, drones, etc.). In such settings, the state-dependency arises from the fact that agents that are physically closer have a higher probability of successfully communicating with each other \cite{lobel2011distributed,alaviani2021distributed, Alaviani:2021b}. Existing results assume that the local interactions between agents lead to strong connectivity over time. Unlike previous work, our model does not assume that the state of an agent represents its the position in space. Moreover, we do not impose strong assumptions on the network's connectivity over time such as in  \cite{nedic2010constrained} and \cite{lobel2011distributed}.

Our work is closely related to state-dependent \textit{averaging} schemes known as \textit{Load-Balancing} \cite{cybenko1989dynamic} and \textit{Greedy Gossip with Eavesdropping} \cite{ustebay2010greedy}. The main idea in these methods is to accelerate averaging by utilizing the information from the most \textit{informative} neighbor, i.e., the neighbors whose states are maximally different with respect to some norm from each agent. We refer to it as the \textit{maximal dissent} heuristics. The challenges of convergence analysis for maximal dissent averaging are highlighted in~\cite{cybenko1989dynamic,nedic2009distributedaveraging,ustebay2010greedy}. However, concepts akin to max-dissent have only been explored for the specific problem of averaging \cite{ustebay2010greedy}. Our work, on the other hand, focuses on distributed convex optimization, whose convergence is not guaranteed by the convergence of the averaging scheme alone. 

As a broader impact of the results herein, we show that schemes that incorporate mixing of information between max-dissent agents will converge to a global optimizer of the underlying distributed optimization problem almost surely. Our result enables us to propose and extend the use of load-balancing, and max-dissent gossip  
to distributed optimization. The key property of max-dissent averaging is that it leads to a contraction of the Lyapunov function used to establishing convergence. While recent work has considered similar contraction results (e.g. \cite{koloskova2019decentralized, koloskova2020unified}), they are not applicable to state-dependent schemes, and do not establish almost sure convergence, but only convergence in expectation.

\subsection{Contributions}

A preliminary version of some of the ideas herein has previously appeared in \cite{Verma:2021}, which addressed only one of the schemes, namely \textit{Global Max-Gossip} for distributed optimization of univariate functions. The results reported here are much more general than the ones in \cite{Verma:2021} addressing $d$-dimensional optimization, and covering multiple state-dependent schemes (e.g. Load-Balancing, Global and Local-Max Gossip) in which the max-dissent agents communicate with non-zero probability at each time. The main contributions of this paper are:
\begin{itemize}
    \item We present state-dependent distributed optimization schemes that do not rely on or imply explicit strong connectivity conditions (such as $B$-connectivity).
     \item We characterize a general result highlighting the importance of max-dissent agents on a graph for distributed optimization, significantly simplifying the task of establishing contraction results for a large class of consensus-based subgradient methods.
    \item We prove the convergence of state-dependent algorithms to a global optimizer for distributed optimization problems using a technique involving the aforementioned contraction property of a quadratic Lyapunov function.
      \item We present numerical experiments that suggest that the proposed state-dependent subgradient methods improve the convergence rate of distributed estimation problems relative to conventional (state-independent) gossip algorithms.
\end{itemize}

\subsection{Organization}

The rest of the paper is organized as follows. First, we formulate distributed optimization problems, and outline a generic state-dependent distributed subgradient method in Section~\ref{sec:problem}. In Section~\ref{sec:consensus} we introduce Local and Global Max-Gossip, and review Randomized Gossip and Load Balancing distributed averaging schemes. We discuss the role of maximal dissent agents and their selection in averaging algorithms in Section~\ref{sec:maxedge}. In Section~\ref{sec:conv}, we present our main results on the convergence of maximal dissent state-dependent distributed subgradient methods. We provide a numerical example that shows the benefit of using algorithms based on the maximal dissent averaging in Section~\ref{sec:numerical}. We conclude the paper in Section~\ref{sec:conclude} where we outline future research directions.

\subsection{Notation}
For a positive integer $n$, we define $[n]=\{1,2,\dots,n\}$. We denote the $d$-dimensional Euclidean space by $\R^d$. 
We use boldface letters such as $\mathbf{x}$ to represent vectors, and lower-case letters such as $x$ to represent scalars. Upper-case letters such as $A$  represent matrices. We use $A^T$ to denote the transpose of a matrix $A$. For $i\in [n]$,  we denote by $\vec{b}_i$ the $i$-th standard basis vector of $\R^n$. We denote by $\one$, the vector with all components equal to one, whose dimension will be clear from the context.
For a vector $\vec{v}$, we denote the $\ell_2$-norm by $\|\vec{v}\|$, and the average of its entries by $\bar{v}$. 
We say that an $n\times n$ matrix $A$ is stochastic if it is non-negative and the elements in each of its rows add up to one. We say that $A$ is doubly stochastic if both $A$ and $A^T$ are stochastic. 
For two vectors $\vec{a}, \vec{b} \in \mathbb{R}^n$, we define $\inp{\vec{a}}{\vec{b}} = \vec{a}^T\vec{b}$. The trace of a square matrix $A$ is defined to be the sum of entries on the main diagonal of $A$ and is denoted by $\Tr(A)$. For matrices $A,B \in \R^{n \times m}$ we define $\inp{A}{B} =\Tr(A^TB)$ as the inner product and $\|A\|_F$ to denote the Frobenius norm of $A$. 
\section{Problem Formulation}\label{sec:problem}
Consider distributed system with $n$ agents with an underlying communication network defined by a graph $\Gcal = (\Vcal, \Ecal)$. Each agent $i \in \Vcal$ has access to a {\em local} convex function $f_i:\R^d \to \R$. 
The agents can communicate only with their one-hop neighbours as dictated by the network graph $\Gcal$. Our goal is to design a distributed algorithm to solve the following unconstrained optimization problem
\begin{align}\label{eq:def_DistOpt_Problem}
 F^* =    \min_{\vecw\in\R^d}  F(\vecw), \;\; \mbox{where} \;\;  F(\vecw) \triangleq \sum_{i=1}^n f_i(\vecw).
\end{align}

We assume that the local objective function $f_i$ is known only to node $i$ and the nodes can only communicate by exchanging information about their local estimates of the optimal solution.

The solution set of the problem is defined as 
\begin{equation}
\sets{W}^* \triangleq \arg \min_{\vecw \in \R^d} F(\vecw). \nonumber
\end{equation}

Throughout the paper, we  make extensive use of the notion of the \textit{subgradient} of a function. 
\begin{definition}[Subgradient]\label{def:subgradient}
A subgradient of a convex function $f:\R^d \to \R$ at a point $\vecw_0\in \R^d$ is a vector $\vec{g} \in \R^d$ such that
\begin{align}
    f(\vecw_0) + \inp{\vec{g}}{\vecw-\vecw_0}  \leq f(\vecw)\nonumber 
\end{align}
for all $\vecw\in \R^d$. We denote the set of all subgradients of $f$ at $\vecw_0$ by $\partial f(\vecw_0)$, which is called the \textit{subdifferential} of $f$ at $\vecw_0$.
\end{definition}

We make the following assumptions on the structure of the optimization problem in \cref{eq:def_DistOpt_Problem}.

\begin{assumption}[Non-empty solution set]\label{asmpn:Nonempty_solution}
The optimal solution set is non-empty, i.e.,
$
     \sets{W}^* \neq \emptyset.
$
\end{assumption}

\begin{assumption}[Bounded Subgradients]\label{asmpn:bounded_subgrad}
Each local objective function $f_i$'s subgradients are uniformly bounded. In other words, for each $ i \in [n] $, there exists a finite constant $L_i$ such that for all $ \vecw\in \R^d $, we have 
 $\|\vec{g}\| \leq L_i, \ \vec{g} \in \partial f_i(\vecw).$
\end{assumption}

There exist many algorithms to solve the problem in \cref{eq:def_DistOpt_Problem}. Nedic and Ozdaglar~\cite{nedic2009distributed} introduced one of the pioneering schemes, in which 
each agent keeps an estimate of the optimal solution 
and at each time step, the agents share their local estimate with their neighbors. Then, each agent updates its estimate using a time-varying, \textit{state-independent} convex combination of the information received from their neighbors and its own local estimate. For $t\geq 0$, let $a_{ij}(t)$ denote the coefficients of the aforementioned convex combination at time $t$ such that $a_{ij}(t)\geq 0$, for all $i,j \in [n]$, $a_{ij}(t)=0$ if $\{i,j\} \notin \mathcal{E}$, and $\sum_{j=1}^n a_{ij}(t)=1$, for all $i\in [n]$.  Let $\mathbf{x}_i(t)$ denote the $i$-th agent's estimate of the optimal solution at time $t$. 
The convex combination is followed by taking a step in the direction of any subgradient in the subdifferential at the local estimate, i.e.,
\begin{align}\label{eq:oldupdateSubgrad}
    \vecx_i(t+1) &=& \sum_{j=1}^n a_{ij}(t)\vecx_j(t) - \alpha(t)\vec{g}_i(t), 
\end{align}
where $\vec{g}_i(t)\in \partial f_i\big(\vecx_i(t)\big)$, 
and $\alpha(t)$ is a step-size sequence.

Herein, we generalize the algorithm in~\cite{nedic2009distributed} by allowing the weights in the convex combination to be \textit{state-dependent} in addition to being time-varying.
Let each agent $i \in \Vcal$ initialize its estimate at an arbitrary point ${\vecx_i(0)\in \R^d}$, which is updated at discrete-time iterations $t\geq 0$ based on its own subgradient and the estimates received from neighboring agents as follows
\begin{align}
    \vecw_i(t+1) &= \sum_{j=1}^n a_{ij}\big(t,\vecx_1(t),\vecx_2(t),\ldots,\vecx_n(t)\big)\vecx_j(t), \nonumber\\ 
    \vecx_i(t+1) &= \vecw_i(t+1) - \alpha(t+1)\vec{g}_i(t+1), \nonumber
\end{align}
where $a_{ij}\big(t,\vecx_1(t),\vecx_2(t),\dots,\vecx_n(t)\big)$ are non-negative weights, $\alpha(t)$ is a step-size sequence, and
$\vec{g}_i(t)\in \partial f_i\big(\vecw_i(t)\big)$ for all $t \geq 0$. We can express this update rule compactly in matrix form as 
\begin{align}
    W(t+1) &= A\big(t,X(t)\big)X(t),  \label{eq:update_mat_1} \\\nonumber
    X(t+1) &= W(t+1) - \alpha(t+1)G(t+1),
\end{align}
where 
\begin{equation}
    A\big(t,X(t)\big) \triangleq \Big[a_{ij}\big(t,\vecx_1(t),\dots,\vecx_n(t)\big)\Big]_{i,j\in [n]}\nonumber
\end{equation}
and 
\begin{equation}
    X(t) \triangleq \begin{bmatrix}
     \vecx^T_1(t) \\
    \vdots \\
     \vecx^T_n(t) 
    \end{bmatrix}, \ 
    W(t) \triangleq \begin{bmatrix}
     \vecw^T_1(t) \\
    \vdots \\
    \vecw^T_n(t)
    \end{bmatrix}, \ 
     G(t) \triangleq \begin{bmatrix}
    \vec{g}^T_1(t)\\
    \vdots \\
    \vec{g}^T_n(t)
    \end{bmatrix}.  \nonumber
\end{equation}
Note that another difference between Eqs. \eqref{eq:update_mat_1}  and \eqref{eq:oldupdateSubgrad} is that agent $i$ computes the subgradient for the local function $f_i$ at the computed average $\vecw_i(t+1)$ instead of
$\vecx_i(t)$, $t\geq 0$.

\begin{assumption}[Diminishing step-size]\label{asmpn:Step-Size}
The step-sizes $\alpha(t)>0$ form a non-increasing sequence that satisfies
\begin{equation}
   \sum_{t=1}^\infty\alpha(t) = \infty \ \ \text{and} \ \ \sum_{t=1}^\infty\alpha^2(t) < \infty. \nonumber
\end{equation}
\end{assumption}
For a  step-size sequence that satisfies \cref{asmpn:Step-Size},
if the sequence of matrices $\{A(t)\}$, where $A(t)=[a_{ij}(t)]_{i,j\in [n]}$, is doubly stochastic and sufficiently mixing, and the objective functions satisfy the regularity conditions in Assumptions  \ref{asmpn:Nonempty_solution} and \ref{asmpn:bounded_subgrad}, 
then the iterates in Eq.~\eqref{eq:oldupdateSubgrad} converge to an optimal solution irrespective of the initial conditions $\vecx_i(0)\in \R^d$, i.e., $
    \lim_{t\to\infty}\vecx_i(t)=\vecx^*, \  i\in [n],
$
where $\vecx^*\in \sets{W}^*$ \cite[Propositions 4 and 5]{nedic2010constrained}. Our goal for the remainder of the paper is to establish a similar result for state-dependent maximal dissent distributed subgradient methods.

\section{State-dependent average-consensus}\label{sec:consensus}
 In this section, we discuss three state-dependent average-consensus schemes that can potentially accelerate the existing distributed optimization methods, {in so doing, we endeavor to unify the state-dependent average-consensus methodology.} The first scheme, Local Max-Gossip, was studied in~\cite{ustebay2010greedy} {exclusively} for the average consensus problem. We provide two novel averaging schemes, the Max-Gossip and Load-Balancing averaging schemes, that provide faster convergence. The dynamics of these algorithms can be understood as the instances of Eq.~\eqref{eq:update_mat_1} with constant local cost functions $f_i(\bfx)\equiv c$, $i\in [n]$, i.e., 
\begin{align}
    X(t+1)=A\big(t,X(t)\big)X(t).\nonumber
\end{align}

We will consider three (two asynchronous and one synchronous) algorithms. The first two algorithms are related to the well-known randomized gossip algorithm \cite{boyd2006randomized,shah2009gossip}. First, we present a brief description of Randomized Gossip.

\subsection{Randomized Gossip}\label{sec:RG}

 Consider a network $\Gcal=([n],\Ecal)$ of $n$ agents, where each agent has an initial estimate 
 $\bfx_i(0)$. At each iteration $t\geq 0$, a  node $i$ is chosen uniformly from $[n]$, independently of the earlier realizations. Then, $i$ chooses one of its neighbors $j\in \Ncal_{i}$, where $\mathcal{N}_{i} \triangleq \{j\in [n]: \{i,j\} \in \Ecal\}$, with probability $P_{ij}>0$. The two nodes exchange their current states $\bfx_{i}(t)$ and $\bfx_{j}(t)$, and update their states according to 
\begin{equation}\label{eq:update}
    \bfx_{i}(t+1)=\bfx_{j}(t+1)=\frac{1}{2}\big(\bfx_{i}(t)+\bfx_{j}(t)\big).
\end{equation}
The states of the remaining agents are unchanged. The update rule in \cref{eq:update} admits a more compact matrix representation as 
\begin{equation}\label{eqn:gossip}
    X(t+1)=\Agossip(e)X(t),
\end{equation}
where $e=\{i,j\}$, and 
\begin{equation}
\Agossip(e)\triangleq I - \frac{1}{2}(\vec{b}_{i}-\vec{b}_{j})(\vec{b}_{i}-\vec{b}_{j})^T. \label{eq:matrix_form_1}
\end{equation}
It is necessary that $\sum_{\ell=1}^nP_{i\ell}=1$ for all $i$, where $P_{i\ell}=0$ if and only if  $\{i,\ell\}\not \in\Ecal$. The dynamical system described in~\cref{eqn:gossip} and its convergence rate are  studied in~\cite{boyd2006randomized}.

\input{global-local}
\subsection{Load-Balancing}

Another state-dependent algorithm known as \textit{Load-Balancing} can also be used to speed up convergence of average-consensus~\cite{nedic2009distributedquantization}. However, in contrast to the previous two cases, where only two nodes update at a given time, Load-Balancing is a synchronous averaging algorithm where all the agents operate simultaneously. 

In the traditional Load-Balancing  algorithm, the state at each agent is a scalar, which induces a total ordering amongst the agents, i.e., the neighbours of an agent are classified by having greater or smaller state values than the agent's current state. When the states at the agents are multi-dimensional vectors, a total ordering is not available and must be defined. 
We introduce a variant of  Load-Balancing based on the Euclidean distance between the states of any two agents as follows.

At time $t$, each agent $i \in [n]$ carries out the following steps:
\begin{enumerate}[1.]
    \item Agent $i$ sends its state to its neighbors. 
    \item Agent $i$ computes the distance between its state and each of its neighbors. Let $\mathcal{S}_i$ denote the subset of neighbors of agent $i$ whose state have maximal Euclidean distance, i.e.,  
    \begin{equation}\label{eq:max_dissent_set}
        \mathcal{S}_i \triangleq \arg \max_{j \in \mathcal{N}_i} \ \|\vecx_i -\vecx_j\|.
    \end{equation}
    Agent $i$ sends an averaging request to the agents in $\mathcal{S}_i.$ \label{step:sendRequest}
  
    \item Agent $i$ receives averaging requests from its neighbors. If it receives a request from a single agent $j\in \mathcal{S}_i$, then it sends an acknowledgement to that agent. In the event that agent $i$ receives multiple requests, it sends an acknowledgement to one of the requests uniformly at random.

    \item If agent $i$ sends and receives an acknowledgement from agent $j$, then it updates its state as $\vecx_i \gets (\vecx_i + \vecx_{j})/2.$ 
\end{enumerate}

\av{The conditions for interaction between two nodes in Load Balancing is characterized in the following proposition.}
\begin{proposition}\label{prop:LBNecessity}
Consider a connected graph $\Gcal$ and a stochastic process $\big\{X(t),A\big(t,X(t)\big)\big\}$, where $A\big(t,X(t)\big)$ is the characterization of averaging according to the Load-Balancing algorithm, i.e. $A\big(t,X(t)\big)X(t)$ is the output of the Load-Balancing algorithm for a network with state matrix $X(t)$, $t \geq 0$. The following statements hold:
\begin{enumerate}
    \item Two agents $i,j$ such that $(i,j) \in \Ecal$ average their states only if
           \begin{multline}\label{eq:LBNecessity}
            \| \vecx_i(t) - \vecx_j(t) \| \geq  \max\Big\{\max_{r\in \Ncal_i\setminus\{j\}}\|\vecx_i(t) - \vecx_r(t)\|,  \max_{r\in \Ncal_j\setminus\{i\}}\|\vecx_j(t) - \vecx_r(t)\|\Big\}.
        \end{multline}
      \item  Let $(i,j) \in \Ecal$. If \cref{eq:LBNecessity} holds with strict inequality, then $i,j$ average their states. 
\end{enumerate}

\end{proposition}
\av{\cref{prop:LBNecessity} is proven in Appendix~\ref{appendix:LBNecessity}.}

\section{On the selection of Max-edges}\label{sec:maxedge}
Consider the stochastic process $\big\{X(t),A\big(t,X(t)\big)\big\}$, where $X(t)$ is the network state matrix, and $A\big(t,X(t)\big)$ a state-dependent averaging matrix. Let $\{\Fcal_t\}_{t=0}^\infty$ be a filtration such that $\Fcal_t$ is the $\sigma$-algebra generated by 
\begin{equation}
    \big\{\{X(k),A\big(k,X(k)\big)\mid k\leq t\big\}\setminus\big\{A\big(t,X(t)\big)\big\}. \nonumber
\end{equation}

We establish a non-zero probability that a pair of agents that constitute a max-edge will update their states for the averaging schemes discussed in \cref{sec:consensus}. 
\begin{proposition}\label{prop:LMLBContracting}
    Let $\big\{X(t),A\big(t,X(t)\big)\big\}_{t=0}^{\infty}$ be the random process generated by either Randomized Gossip, Local Max-Gossip, Max-Gossip, or Load-Balancing consensus schemes. Then, for the random indices $i^*,j^* \in \Vcal$ defined through the max-edge in \cref{eq:emax} as $\emax\big(X(t)\big)=\{i^*,j^*\}$, we have 
    \begin{equation}\label{eqn:maxEdgeDelta}
        \E \Big[A\big(t,X(t)\big)^T \! \! A\big(t,X(t)\big)  \mid \Fcal_t \Big]_{i^*j^*}\geq \delta \ \ \text{a.s.,}
    \end{equation}
    where $\delta = \min_{\{i,j\}\in \Ecal} P_{ij}/n$ for Randomized Gossip, such that $P_{ij}$ is the probability that node $i$ chooses node $j \in \mathcal{N}_j$; $\delta = 1/n$ for Local Max-Gossip; $\delta = 1/2$ for Global Max-Gossip; and $\delta=1/\big(2(n-1)^2\big)$ for Load-Balancing.

\end{proposition}

\av{\cref{prop:LMLBContracting} establishes that given the knowledge until time $t$, in expectation,  the agents comprising the max-edge based on the network state matrix $X(t)$, exchange their values with a positive weight bounded away from zero. Qualitatively, for gossip-based algorithms, this implies that there is a positive probability bounded away from zero that the agents comprising the max-edge carry out exchange of information with each other. 
We use \cref{prop:LMLBContracting} along with \cref{thm:maxEdgeContraction} to establish that the averaging matrices characterizing the algorithms discussed in  \cref{sec:consensus} are contracting. Therefore, the subgradient methods based on these averaging algorithms converge to the same optimal solution almost surely as stated in \cref{cor:ConvergenceExamples} of \cref{thm:main1}. In other words, as long as the averaging step involves gossip over the max-edge with positive probability (bounded away from zero), we will have a contraction in the Lyapunov function capturing the sample variance, which is a key step in proving the convergence of our averaging based-subgradient methods. \cref{prop:LMLBContracting} is proven in Appendix~\ref{appendix:LMLBContracting}.}
\section{Convergence of state-dependent Distributed Optimization
}\label{sec:conv}

In the previous section, we have set the stage for studying the convergence of state-dependent averaging-based distributed optimization algorithms. Our proofs rely on two properties: double stochasticity and the \textit{contraction property} (\cref{thm:maxEdgeContraction}).

To state the contraction property, we define the Lyapunov function $V:\R^{n \times d} \rightarrow \mathbb{R}$ as
\begin{equation} \label{eq:Lyapunov}
    V(X) \triangleq \sum_{i=1}^n\|\vecx_i  - \bar{\vecx}\|^2,
\end{equation}
where $X = [\vecx_1,\cdots,\vecx_n]^T$ and  $\bar{\vecx}= \frac{1}{n}\sum_{i=1}^n\vecx_i$.

\begin{theorem}[Contraction property]\label{thm:maxEdgeContraction}
Consider a connected graph $\Gcal$ and the stochastic process $\big\{X(t),A\big(t,X(t)\big)\big\}_{t=0}^{\infty}$ with a natural 
filtration $\{\Fcal_t\}_{t\geq 0}$ for the dynamics given by Eq.~\eqref{eq:update_mat_1}. 
If ${A\big(t,X(t)\big) \in \Fcal_{t+1}}$ is doubly stochastic for all $t \geq 0$, and for the random variables $i^*,j^* \in \Vcal$ defined through the max-edge in \cref{eq:emax} as $\emax\big(X(t)\big)=\{i^*,j^*\}$,
\begin{equation}\label{ineq:deltaCondition}
    \E\Big[A\big(t,X(t)\big)^T\! \!A\big(t,X(t)\big)\mid \Fcal_t\Big]_{i^*j^*} \geq \delta, \ \ \text{a.s.,}
\end{equation}
where  $\delta>0$, holds for all $t\geq 0$ and $X(0) \in \R^{n\times d},$ then
\begin{equation}\label{eq:thmContraction}
    \E\Big[V\big(A\big(t,X(t)\big)X(t)\big) \mid \Fcal_t\Big] \leq \lambda V\big(X(t)\big) \, \, \text{a.s.},
\end{equation}
where   $\lambda = 1-2\delta/\big((n-1)\diam^2\big).$
\end{theorem}

Theorem~\ref{thm:maxEdgeContraction} is proven in Appendix~\ref{appendix:edgecontraction} and provides our key new ingredient:  proving a contraction result for doubly stochastic averaging matrices containing the maximally dissenting edge. The proof of Theorem~\ref{thm:maxEdgeContraction} makes use of the double stochasticity of the matrices to characterize the exact one-step decrease in the Lyapunov function and then uses a clever trick to characterize its fractional decrease based on the fact that underlying communication graph is connected. 
\begin{remark}
Theorem~\ref{thm:maxEdgeContraction} also holds for time-varying graphs provided they remain connected at each time $t$. More precisely, the theorem holds for a sequence of connected graphs $\{\Gcal_t\}$ and at every time $t\geq 0$, for $i^*,j^*$ defined through $\emax(\Gcal_t,X(t))$, the inequality in \cref{ineq:deltaCondition} holds, then the inequality in \cref{eq:thmContraction} will hold with scaling at time $t$ being 
\begin{equation}
 \lambda_t = 1-\frac{2\delta}{(n-1)\mathrm{diam}({\Gcal_t)}^2} \leq 1 - \frac{2\delta}{(n-1)^3}.  \nonumber
\end{equation} 
Therefore, the contraction property for connected time-varying graphs holds with a factor of at most
$\underline{\lambda} \triangleq 1 - 2\delta/(n-1)^3.$
\end{remark}

For a connected graph $\Gcal$ and the stochastic process $\big\{X(t),A\big(t,X(t)\big)\big\}_{t= 0}^{\infty}$ with the 
filtration $\{\Fcal_t\}_{t= 0}^{\infty}$ generated according to the dynamics in \cref{eq:update_mat_1}, we define a contracting averaging matrix as follows.
\begin{definition}[Contracting  averaging matrix]\label{asmpn:doublystochastic}
A state-dependent averaging matrix  $A\big(t,X(t)\big)$ is contracting with respect to the Lyapunov function $V(\cdot)$ in \cref{eq:Lyapunov} if there exists a $\lambda \in (0,1)$ such that
\begin{equation}\label{eq:defContractingMatrix}
    \E\Big[V\Big(A\big(t,X(t)\big)X(t)\Big) \mid \Fcal_t\Big] \leq \lambda V\big(X(t)\big) 
\end{equation}
holds a.s. for all $t\geq 0$. 
\end{definition}

{ The main result of this work establishes convergence guarantees for these dynamics as stated below.} 

\begin{theorem}[Almost sure convergence of state-dependent subgradient methods]\label{thm:main1}
Consider the distributed optimization problem in \cref{eq:def_DistOpt_Problem} and let Assumptions \ref{asmpn:Nonempty_solution} and \ref{asmpn:bounded_subgrad} hold. Assume a connected communication graph $\Gcal$ and the subgradient method in \cref{eq:update_mat_1}. If the random matrices $A\big(t,X(t)\big)$ in Eq.~\eqref{eq:update_mat_1} are 
doubly stochastic and contracting, and the step-sizes $\{\alpha(t)\}$ follow \cref{asmpn:Step-Size}, 
then for all initial conditions $X(0)\in \R^{n\times d}$, 
 \begin{equation}
    \lim_{t \to \infty} \vecw_i(t) = \vecw^*, \ \forall i \in [n], \ \mbox{ a.s.}, \nonumber
 \end{equation}
 where $\vecw^* \in \sets{W}^*$.
\end{theorem}
\cref{thm:main1} establishes the almost-sure convergence of the state variables to an optimal solution of \cref{eq:def_DistOpt_Problem}, based on the consensus-based subgradient methods where the averaging matrices are doubly stochastic and contracting.
\cref{thm:maxEdgeContraction} provides a simplified condition, the presence of averaging over the `max-edge', which, when satisfied, implies the averaging matrix is contracting.  Note that, as shown in \cref{prop:LMLBContracting}, this simplified condition holds for Local Max-Gossip, Max-Gossip, and Load-Balancing averaging. Thus, we have the subsequent corollary following immediately from \cref{prop:LMLBContracting}, \cref{thm:maxEdgeContraction}, and \cref{thm:main1}.

\begin{corollary}\label{cor:ConvergenceExamples}
Consider the distributed optimization problem in \cref{eq:def_DistOpt_Problem} and let Assumptions \ref{asmpn:Nonempty_solution} and \ref{asmpn:bounded_subgrad} hold. Assume a connected communication graph $\Gcal$ and the subgradient method \eqref{eq:update_mat_1} where the averaging matrices $A\big(t,X(t)\big)$ in Eq.~\eqref{eq:update_mat_1} are based solely on either the Local Max-Gossip, Max-Gossip or Load-Balancing averaging, and the step-sizes $\{\alpha(t)\}$ follow \cref{asmpn:Step-Size}. Then 
 \begin{equation}
    \lim_{t \to \infty} \vecw_i(t) = \vecw^*, \,\,  \forall i \in [n], \ \ \text{a.s.}, \nonumber
 \end{equation}
 for all initial condition $X(0)\in \R^{n\times d}$, and some $\vecw^* \in \sets{W}^*$.
\end{corollary}
For the remainder of this section, we provide the key steps and  results that are needed to prove \cref{thm:main1}. We defer the proof of these technical results to the Appendix. 

The proof strategy for Theorem~\ref{thm:main1} can be broken down into two main steps: (i) showing that the evolution of the dynamics followed by the average state variable $\{\bar{\vecx}(t)\}$ converges to a solution of the optimization problem in \eqref{eq:def_DistOpt_Problem} and (ii) every node $i \in \Vcal$ tracks the dynamics of this average state variable such that the tracking error goes to zero.
The first step requires the following result which establishes a bound on the accumulation of the tracking error for every agent. 

\begin{lemma}\label{lemma:dist_from_mean}
Let $\Gcal$ be a connected graph and consider sequences $\{W(t)\}$ and $\{X(t)\}$ generated by the subgradient method in Eq.~\eqref{eq:update_mat_1} using sate-dependent, doubly stochastic and contracting averaging matrices $A\big(t,X(t)\big)$. If Assumptions~\ref{asmpn:bounded_subgrad} and  \ref{asmpn:Step-Size} hold, then for any initial estimates $X(0)\in \mathbb{R}^{n\times d}$, the following hold  a.s.\ for all $i\in [n]$
 \begin{align}
     \lim_{t \to \infty} \left\|\vecw_i(t+1) - \bar{\vecx}(t)\right\| = 0,\qquad\text{and}\cr
     \sum_{t=0}^{\infty}\alpha(t+1)\E\left[ \|\vecw_i(t+1) - \bar{\vecx}(t) \| \mid \Fcal_t\right]<\infty. \nonumber
\end{align}  
\end{lemma}

Lemma~\ref{lemma:dist_from_mean}, which is proven in Appendix~\ref{appendix:distfrommean}, establishes guarantees on the consensus error for the local estimates $\vecw_i(t)$. Lemma~\ref{lemma:iterate_decomp} will be used to bound the distance of the average state $\bar{\vecx}(t)$  to an optimal point. 

\begin{lemma}[Lemma 8,~\cite{Nedic_2013}]\label{lemma:iterate_decomp}
Suppose that Assumption \ref{asmpn:bounded_subgrad} holds. Then, for any connected graph $\Gcal$, initial condition $X(0)\in \R^{n\times d}$, $\vec{v}\in \R^d$, and $t \geq 0$, for the dynamics $\{X(t),A(t,X(t))\}$ of the subgradient method Eq.~\eqref{eq:update_mat_1} where $A(t,X(t))$ are doubly stochastic, 
we have 
\begin{multline}
    \E\big[\|\bar{\vecx}(t+1) - \vec{v}\|^2 \mid \Fcal_t \big]  
    \leq  
    \|\bar{\vecx}(t) - \vec{v}\|^2 -  \alpha(t+1)\frac{2}{n}\Big(F\big(\bar{\vecx}(t)\big) - F(\vec{v})\Big) \\
    +\alpha(t+1)\frac{4}{n}\sum_{i=1}^n L_i \E[\|\vecw_i(t+1) \textbf{}- \bar{\vecx}(t)\|\mid \Fcal_t]  +\alpha^2(t+1)\frac{L^2}{n^2}, \ \  \text{a.s.} \nonumber
\end{multline}
\end{lemma}

We note that Lemma 8 in~\cite{Nedic_2013} was originally intended for state independent dynamics. However, its proof only relies on the double stochasticity of the averaging matrices, convexity of the local functions, boundedness of the subgradients, and not on whether the averaging is state-dependent or not. 
Finally, combining the above two results implies that the distance of each agent's local  estimate $\vecx_i(t)$ to the optimal set $\sets{W}^*$ will be \textit{approximately} decreasing. The following result then will be used to show that this \textit{approximate} decrease results in convergence to $\sets{W}^*$. 
\begin{lemma}
\label{lemma:convex}
Consider a minimization problem $\min_{\vecx \in \R^d} f(\vecx)$, where $f:\R^d \to \R$ is a continuous function. Assume that the solution set $\sets{X}^*$ of the problem is nonempty. Let $\{\vecx_t\}$ be a stochastic process such that for all $\vecx \in \sets{X}^*$ and for all $t \geq 0$, 
\begin{equation}
    \E[\|\vecx_{t+1} - \vecx\|^2 \mid \Fcal_t] 
    \leq (1+b_t)\|\vecx_t - \vecx\|^2  - a_t\big(f(\vecx_t) - f(\vecx)\big) + c_t \qquad\textrm{a.s.}, \nonumber
\end{equation}
where $b_t\geq 0$, $a_t\geq 0$, and $c_t \geq 0$ for all $t \geq 0$ and ${\sum_{t=0}^\infty b_t <\infty} $, $\sum_{t=0}^\infty a_t = \infty $, and $\sum_{t=0}^\infty c_t < \infty$ a.s. Then the sequence $\{\vecx_t\}$ converges to a solution 
$\vecx^* \in \sets{X}^*$ a.s.
\end{lemma}
This result has been proven as part of~\cite[Theorem~1]{aghajan2020distributed} but due to the stand-alone significance of the result we have stated it as a lemma above and its proof is provided in Appendix~\ref{appendix:lemmaConvex}.
Now, we are ready to formally prove Theorem~\ref{thm:main1} by combining the aforementioned results.\\
\begin{proof}[{Proof of Theorem~\ref{thm:main1}}]
From Lemma~\ref{lemma:iterate_decomp}, for ${\vec{v} = \vecw^*\in \sets{W}^*}$, we have
\begin{multline}
    \E[\|\bar{\vecx}(t+1) - \vecw^*\|^2 \mid \Fcal_t]\leq  
      \|\bar{\vecx}(t) - \vecw^*\|^2  - \frac{2\alpha(t+1)}{n}\Big(F\big(\bar{\vecx}(t)\big) - F^*\Big) +  \alpha^2(t+1)\frac{L^2}{n^2} \\ +4\frac{\alpha(t+1)}{n}\sum_{i=1}^n L_i \E[\|\vecw_i(t+1) - \bar{\vecx}(t)\|\mid \Fcal_t],\nonumber
\end{multline}
for all $t \geq 0$. 
From Lemma~\ref{lemma:dist_from_mean}, we know that 
\begin{align}
    &\sum_{t=0}^{\infty} 4\frac{\alpha(t+1)}{n}\sum_{i=1}^n L_i \E[\|\vecw_i(t+1) - \bar{\vecx}(t)\|\mid \Fcal_t] = \nonumber
    \\\nonumber
    &\sum_{i=1}^n \frac{4L_i}{n} \sum_{t=0}^{\infty}\alpha(t+1)\E[\|\vecw_i(t+1) - \bar{\vecx}(t)\| \mid \Fcal_t]< \infty  \ \ a.s.  
\end{align}
Furthermore, $\alpha(t)$ is not summable and ${\sum_{t=0}^\infty \alpha^2(t) < \infty}$. Therefore, all the conditions for Lemma~\ref{lemma:convex} hold with  $a_t = 2\alpha(t+1)/n$, $b_t = 0$, and 
\begin{equation}
c_t = \alpha(t+1)\frac{4}{n}\sum_{i=1}^n L_i\E[ \|\vecw_i(t+1) - \bar{\vecx}(t)\| \mid \Fcal_t ]  + \alpha^2(t+1)\frac{L^2}{n^2}. \nonumber
\end{equation}

Therefore, from Lemma~\ref{lemma:convex}, the sequence $\{\bar{\vecx}(t)\}$ converges to a solution $\hat{\vecw} \in \sets{W}^*$ almost surely. Finally, Lemma~\ref{lemma:dist_from_mean} implies that $\lim_{t \to \infty} \|\vecw_i(t+1)- \bar{\vecx}(t)\| = 0$ for all $i \in [n]$ almost surely. Therefore, the sequences $\{\vecw_i(t+1)\}$ converge to the same solution $\hat{\vecw}\in \sets{W}^*$ for all $i \in [n]$ almost surely. 
\end{proof}

\input{convergence_rate}

\section{Numerical Examples}\label{sec:numerical}
To illustrate our analytical results, we present a simulation of a distributed optimization problem where the local functions' subgradients are not restricted to be uniformly bounded. In particular, we look at the standard distributed estimation problem in a  sensor network setting with $n=180$ agents. Here, each agent $i \in \Vcal $ wants to estimate an unknown parameter $\theta_0$. 
Each node has access to a noisy measurement of the parameter $c_i = \theta_0 + n_i$, where $n_i$'s are independent, zero mean Gaussian random variables with variance $\sigma^2_i>0$. In this setting, the Maximum Likelihood (ML) estimator $\text{\cite[Theorem~5.3]{van2007parameter}}$ 
is the minimizer of the separable cost function $F(w) = \sum_{i=1}^n (w - c_i)^2/\sigma_i^2$. Note that this problem is a distributed optimization problem with the local cost function $f_i(w)=(w - c_i)^2/\sigma_i^2$. For the variance $\sigma_i^2$, we picked $1/\sigma_i^2$ independently  and uniformly over $(0,1)$. For each node $i \in [n]$, the initial local estimates $x_i(0)$ are drawn independently from a standard Gaussian distribution.

We consider the performance for different topologies of the underlying communication graph $\mathcal{G}$ ranging from dense graphs (Complete and Barbell),
moderately dense graphs (Erd\"{o}s-R\'enyi),
to sparse graphs (Line and Star).
We chose a connected graph with the edge probability $p=0.4$ for Erd\"{o}s-R\'enyi graph. For the Barbell graph, we chose equal number of nodes for the three components -- two Complete graphs and the connecting Line graph.

We ran the averaging-based subgradient optimizer with four different averaging update rules: Randomized Gossip \cite{boyd2006randomized}, Local Max-Gossip, Max-Gossip, and Load-Balancing. For the Randomized Gossip, at each time a node in $[n]$ wakes up uniformly at random, and it chooses one of its neighbors uniformly at random for communication. To account for the stochastic nature of Randomized Gossip and Local Max-Gossip algorithms we average the error values over $10$ runs keeping the initial conditions and samples at the nodes the same. The resulting plots in \cref{fig:experiment180}, show the decay of the error $\|\vecw(t) - w_*\one\|$ as a function of $t$, where $w_*=\sum_{i=1}^n \frac{c_i}{\sigma_i^2}/\sum_{i=1}^n \frac{1}{\sigma_i^2}$ is the optimal solution for $F(w)$. \av{For the Erd\"{o}s-R\'enyi communication graph, we also plot the decay of the error with the number of bits exchanged between the nodes in \cref{fig:er_plots} for Randomized Gossip, Local Max-Gossip, and Load-Balancing.} 

\av{In the simulation, $32$ bits are used for exchange of the estimates and $1$ bit is used for the exchange of each acknowledgement. Therefore, the number of bits exchanged per step for Randomized Gossip is $64$. For Local Max-Gossip, at time $t$ with $s(t)\in [n]$ being the randomly chosen node, $|\Ncal_{s(t)}| + 32 |\Ncal_{s(t)}| +  32 $ bits are exchanged for waking up the neighboring nodes, obtaining their values, and sending the neighbor with the maximum disagreement its own value. Finally, for Load Balancing, $ 32\sum_{i=1}^n |\Ncal_i| + n + {\text{ACK}}(t)$ bits are exchanged for sharing the values with the neighbors, sending request to the neighbour with the maximum disagreement, and sending the acknowledgement, where ${\text{ACK}}(t)$ is the total number of bits exchanged for sending the acknowledgement bits at time $t$.\footnote{In the numerical simulation, there are no cases with multiple neighbors with maximum disagreement.}}

\subsection{Comparison of Asynchronous Methods}
From \cref{fig:experiment180}, the performance of the subgradient methods using state-dependent averaging shows an improvement in convergence rate. The convergence rates increase as we go from Randomized Gossip, Local Max-Gossip, Max-Gossip to Load-Balancing averaging based optimizers. We will refer to the subgradient methods using the state-dependent averaging by their averaging algorithm in the succeeding discussion.  

In general, the performance of Max-Gossip is superior to the one of Local Max-Gossip. 
Clearly, Local Max-Gossip converges faster than Randomized Gossip. However convergence rate also depends on the graph topology: Local Max-Gossip applied on a Star graph has essentially the same rate as Randomized Gossip since the nodes at the periphery have only the central node as the choice to gossip with, and the probability of the first node being selected for gossiping is $n-1$ times larger to be a peripheral node as compared to the central node.
Overall, we notice the increase in the performance of Max-Gossip and Local Max-Gossip as compared to Randomized Gossip with increasing connectivity.  \av{Moreover, from \cref{fig:er_plots} we note the significantly better performance of Local Max-Gossip with respect to the number of exchanges between the nodes as opposed to that of synchronous Load-Balancing. }

\subsection{Max-Gossip vs. Load-Balancing}
When comparing different state-dependent averaging schemes, it should be noted that unlike gossip, Max-Gossip, and  Local Max-Gossip, Load-Balancing is a synchronous scheme where in addition to the max-edge, 
 other local max-edges are often incorporated in the averaging scheme simultaneously.  
Therefore, it is only natural that the convergence rate of Load-Balancing   is superior to that of Max-Gossip, since it averages not only the two nodes defined by the max-edge, but, additionally, other nodes connected by edges with large disagreement at the same time. 
By a similar logic, for the Complete graph, the performance of Load-Balancing   and Max-Gossip are the same since all the nodes are holding scalar estimates and due to the ordering between the estimates, all the nodes send their request for averaging to either the node with the maximum or minimum estimate resulting in  only the max-edge performing the updates.

We observe that the gap in performance of Load-Balancing   and Max-Gossip, which has the best performance amongst the discussed asynchronous methods, increases with the diameter of the graph. Characterizing the analytical dependence of convergence rate as a function of graph topology metrics is of interest for future work.

\begin{figure}[t] 
\centering
\begin{subfigure}{0.45\textwidth}
    \centering
    \includegraphics[width=\textwidth]{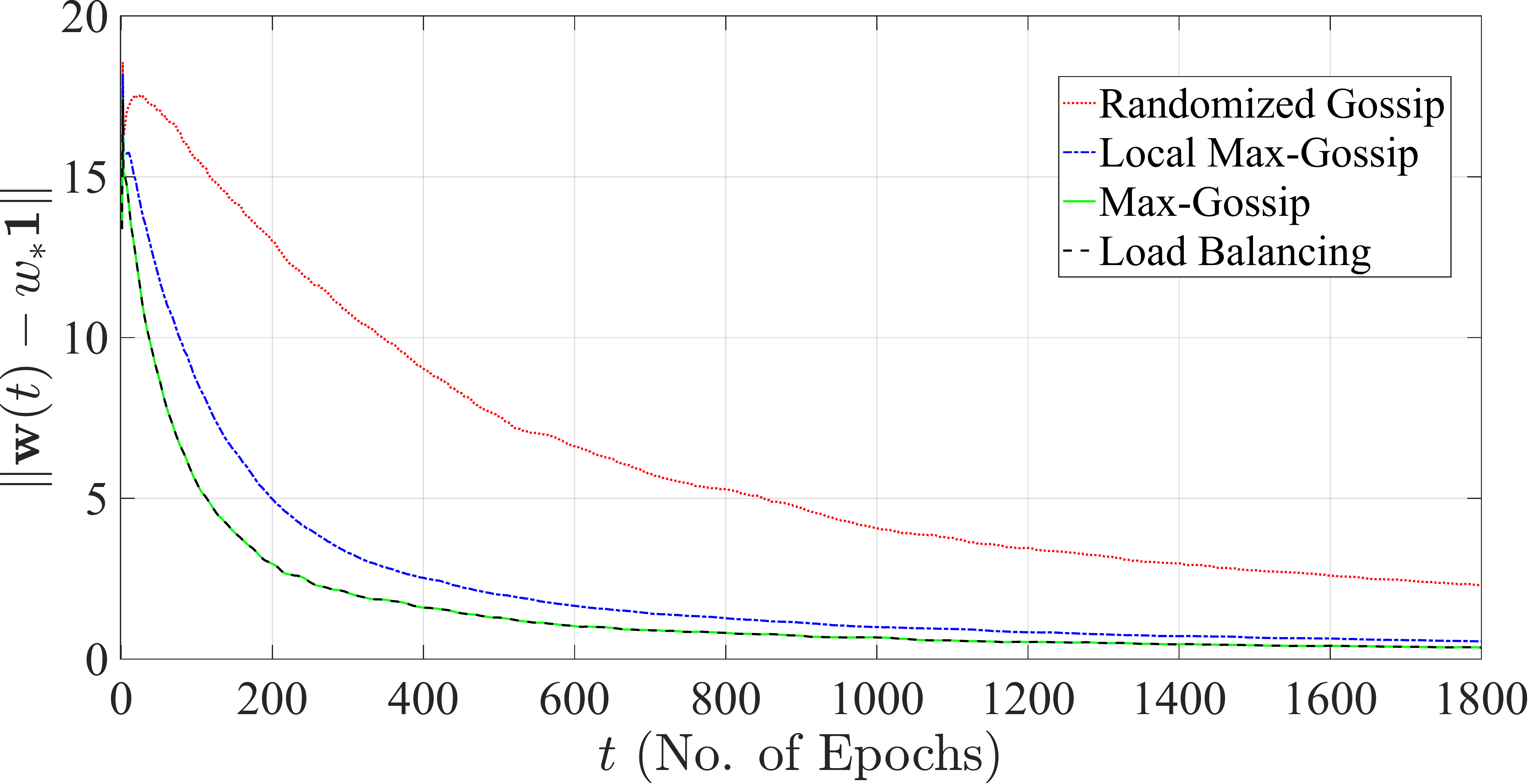}
    \subcaption{Complete Graph}
\end{subfigure}
\begin{subfigure}{0.45\textwidth}
    \centering
    \includegraphics[width=\textwidth]{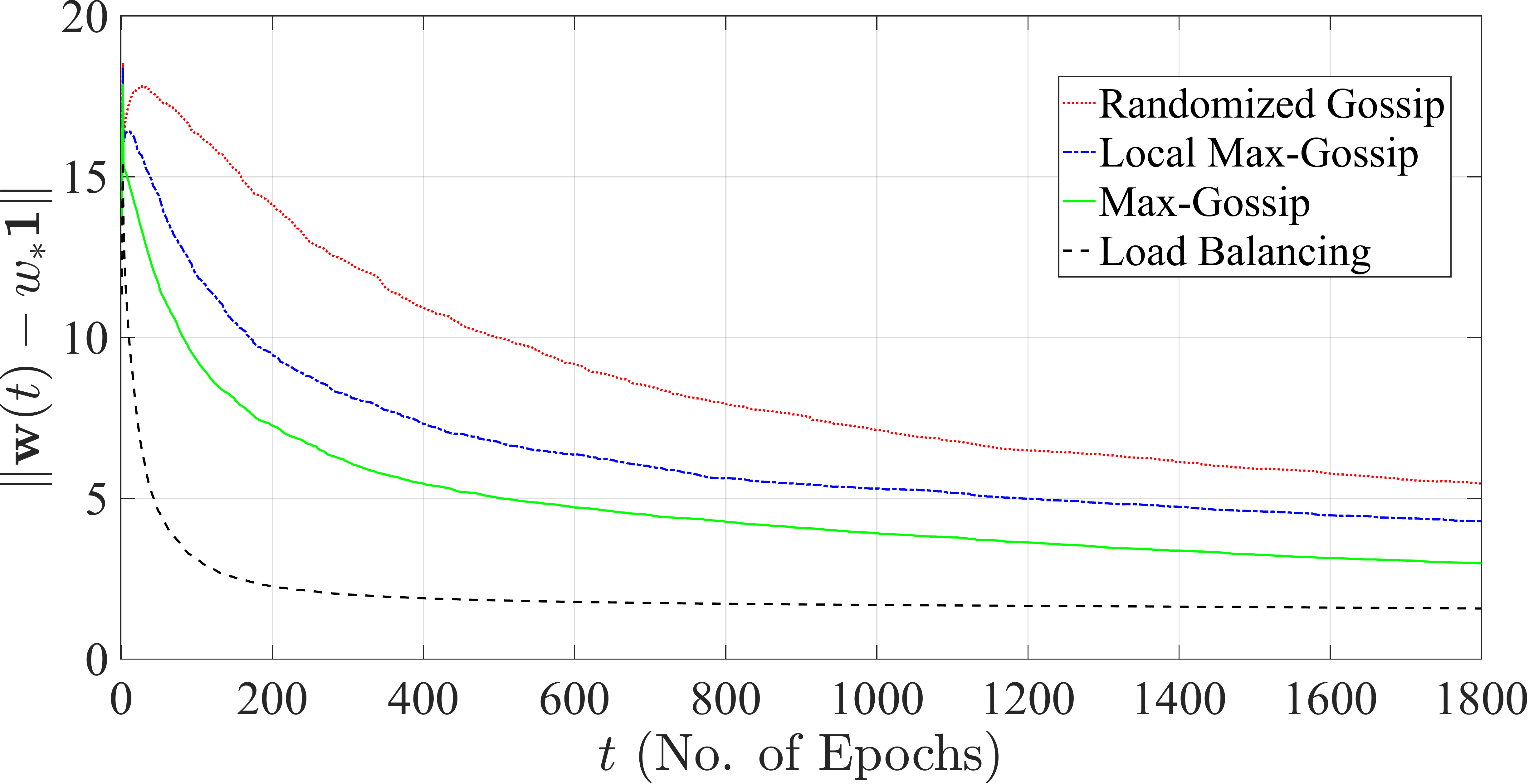}
    \subcaption{Barbell Graph}
\end{subfigure}
\begin{subfigure}{0.45\textwidth}
     \centering
    \includegraphics[width=\textwidth]{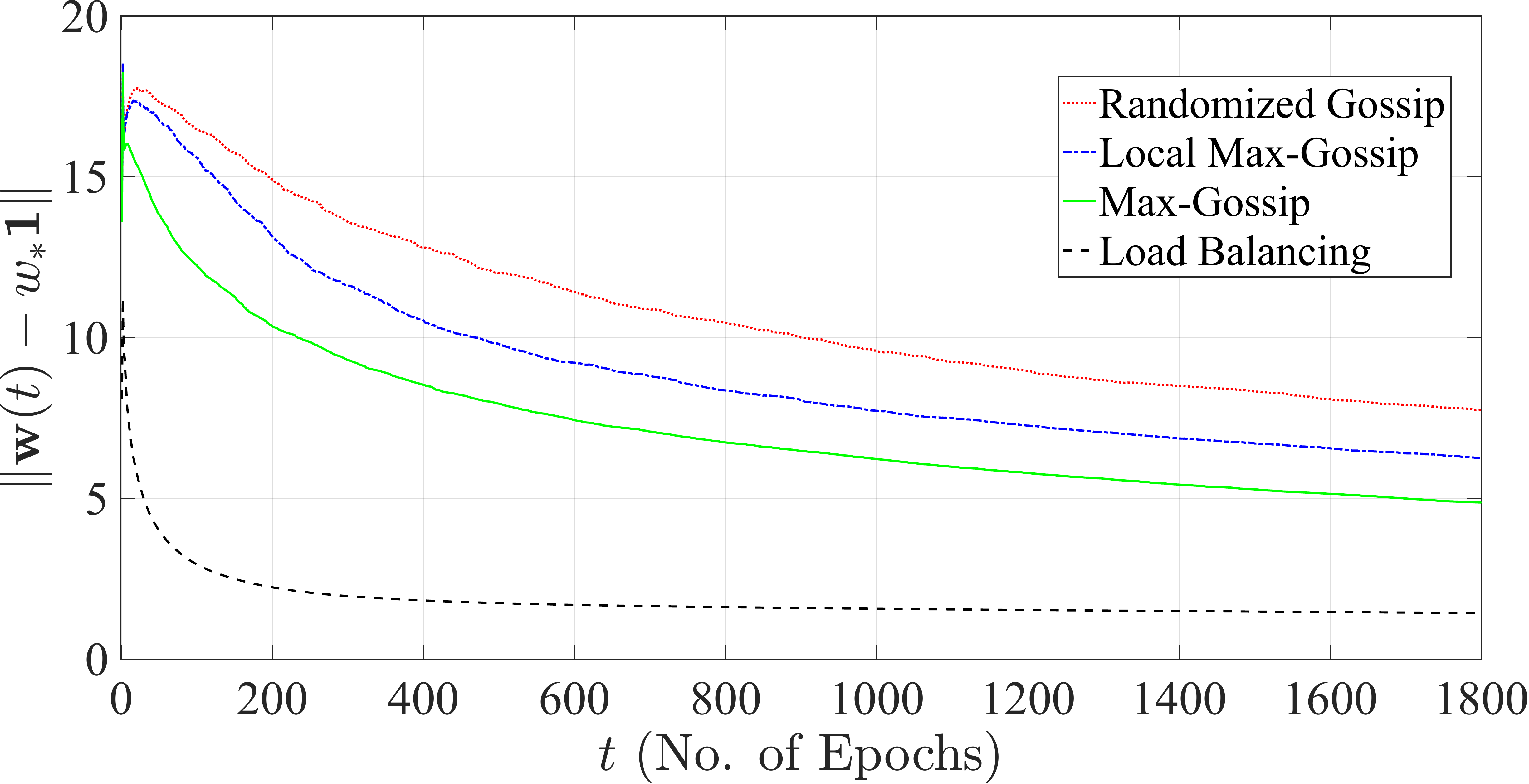}
\subcaption{Line Graph}
\end{subfigure}
\begin{subfigure}{0.45\textwidth}
     \centering
    \includegraphics[width=\textwidth]{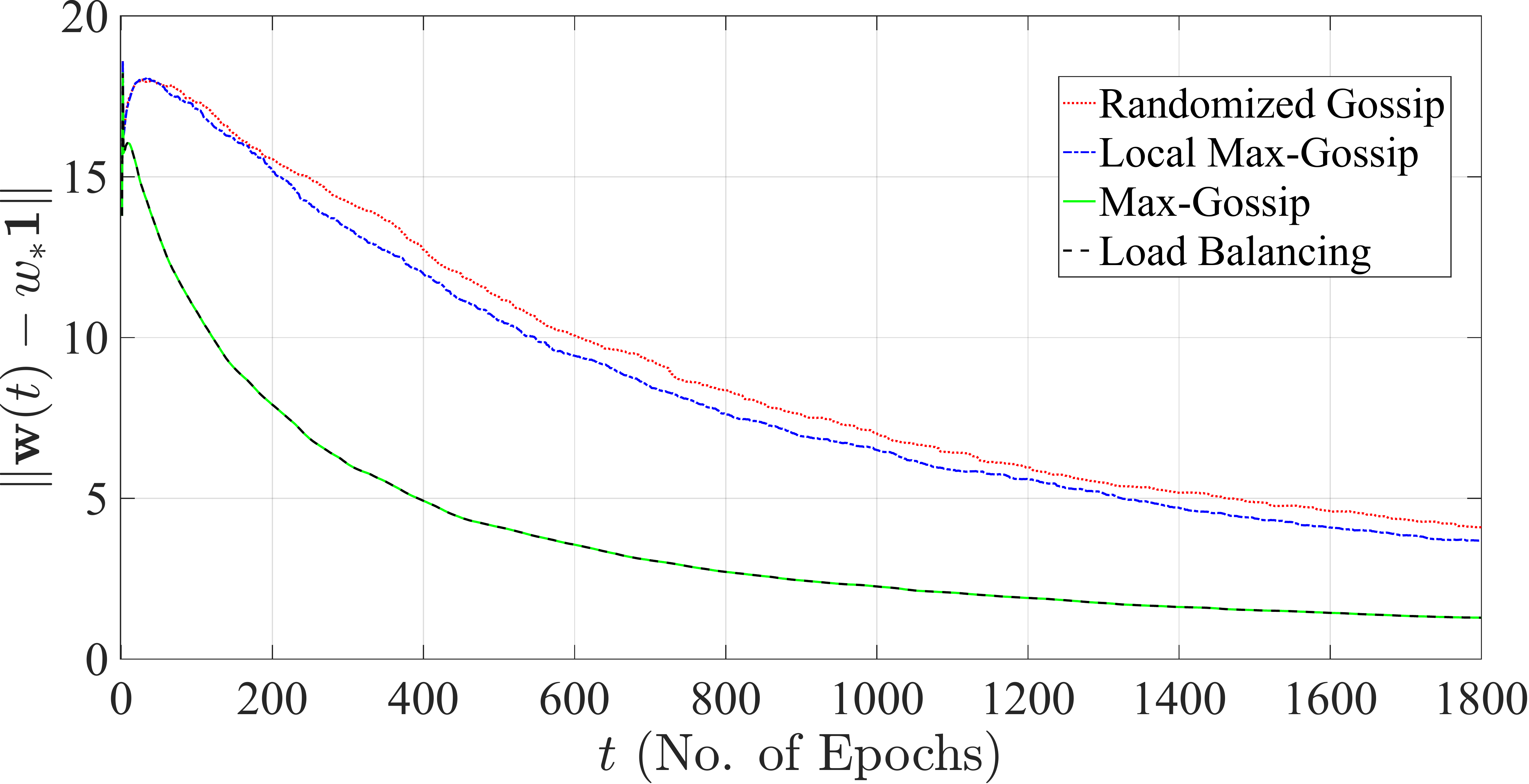}
\subcaption{Star Graph}
\end{subfigure}
\caption{Error decay for different graphs with $180$ nodes}
\label{fig:experiment180}
\end{figure}

\begin{figure}
    \centering
    \begin{subfigure}{0.45\textwidth}
    \centering
    \includegraphics[width=\textwidth]{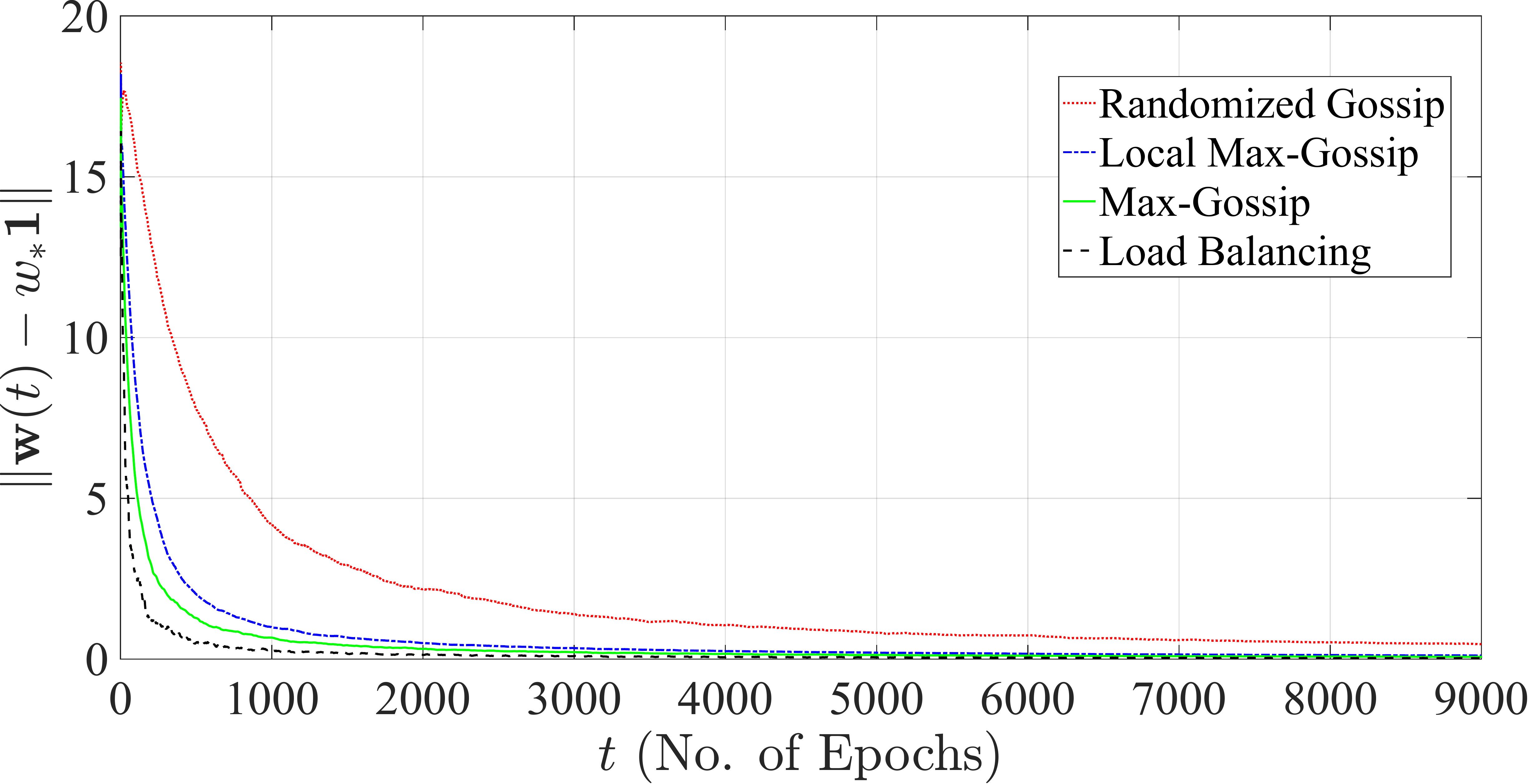}
\end{subfigure}
\begin{subfigure}{0.45\textwidth}
     \centering
    \includegraphics[width=\textwidth]{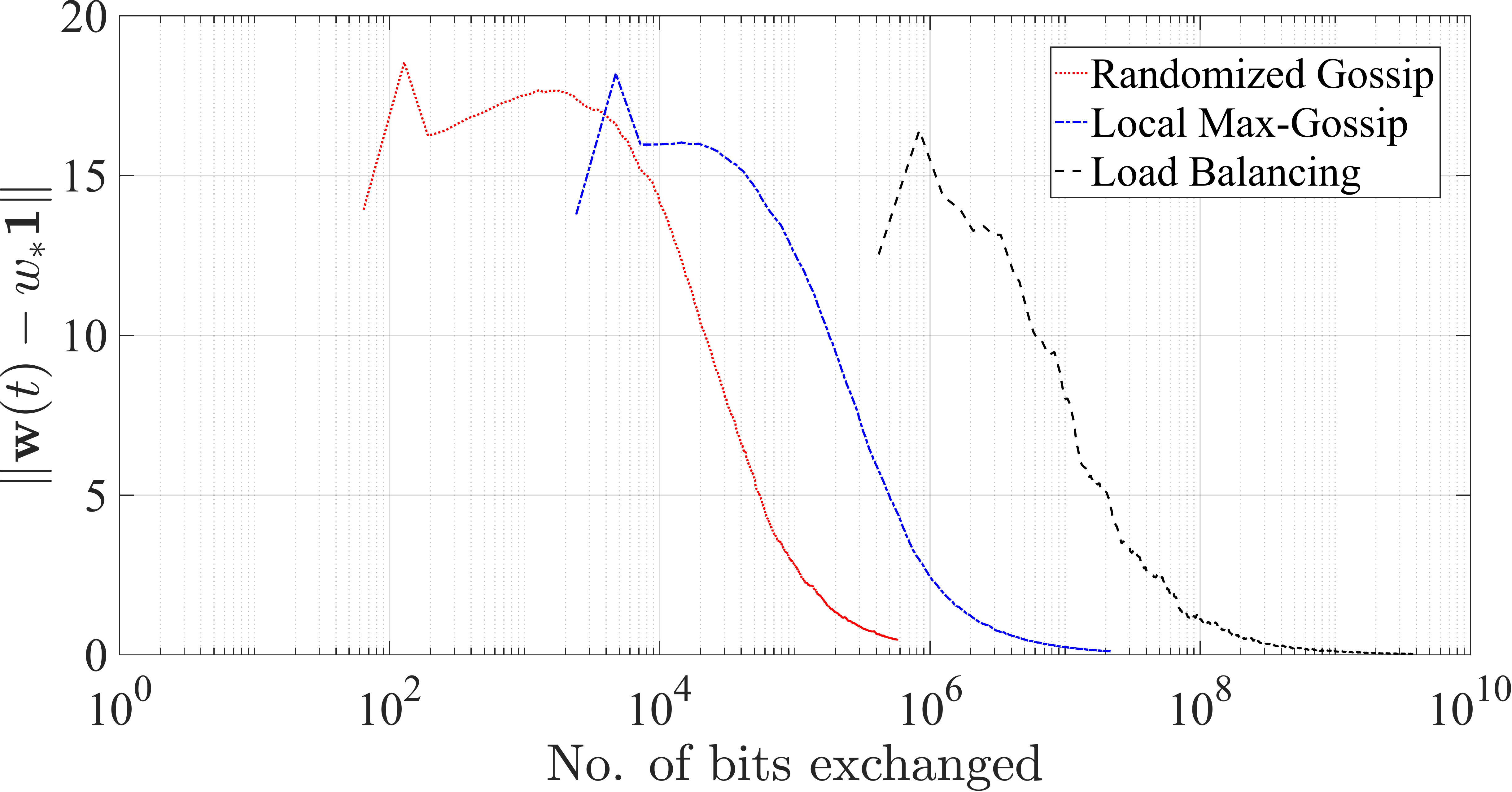}
\end{subfigure}
\caption{Error decay for Erd\"{o}s-R\'enyi Graph with $180$ nodes}
\label{fig:er_plots}
\end{figure}

\color{black}
\subsection{Logistic Regression}

In order to illustrate the applicability of the results to a more general high-dimensional convex problem, we look at an example of regularized logistic regression for classification over MNIST dataset containing $56000$ samples. In the experiment we train a model with the loss function defined as
\begin{align*}
    J(\vec{w},b) &= \frac{1}{m}\sum_{j=1}^m \left(-y_j\log\frac{1}{1+\exp(-(\vec{x_j}^T\vec{w})+b)} -(1-y_j)\log\frac{\exp(-(\vec{x_j}^T\vec{w})+b)}{1+\exp(-(\vec{x_j}^T\vec{w})+b)}\right) \cr 
&\qquad + \frac{1}{2m}\|\vec{w}\|^2+ \frac{1}{2m}|b|^2,
\end{align*}
where $\{(\vec{x}_j,y_j)\}_{j=1}^{56000}$ are the samples used for training.
The samples are used to classify the digits in MNIST dataset into two classes based on whether the digits are greater than or equal to $5$ or not. The experiment is run over a graph with $20$ nodes with each node containing the same number of samples from the dataset. We initialize the nodes with all zero vectors. 

The communication graph representing the underlying connection between the nodes is a ladder graph. We consider the performance for the averaging-based subgradient optimizer with Randomized Gossip, Local Max-Gossip, Max-Gossip, and Load Balancing. For Randomized Gossip, as in previous experiment, at each time a node in $[n]$ wakes uniformly at random and chooses one of its neighbor uniformly at random. We average the performance for Randomized Gossip and Local Max-Gossip over 3 runs. In \cref{fig:network_variance} we plot the network variance, $\|W(t) - \frac{\one\one^T}{n}W(t)\|_F^2$ for step-size $\alpha(t)=\frac{1}{t}$ for all $t\geq 1$. We observe that the decay in the loss of the function for the consensus-based subgradient method is similar to each other. However the decay of the network variance, defined as the sum of the square of the deviation of the state estimates from their mean, over time in decreasing order of speed is observed for Load Balancing, Max-Gossip, Local Max-Gossip, and finally Randomized Gossip.

\begin{figure}
\begin{subfigure}{0.5\textwidth}
    \centering
    \includegraphics[width=\textwidth]{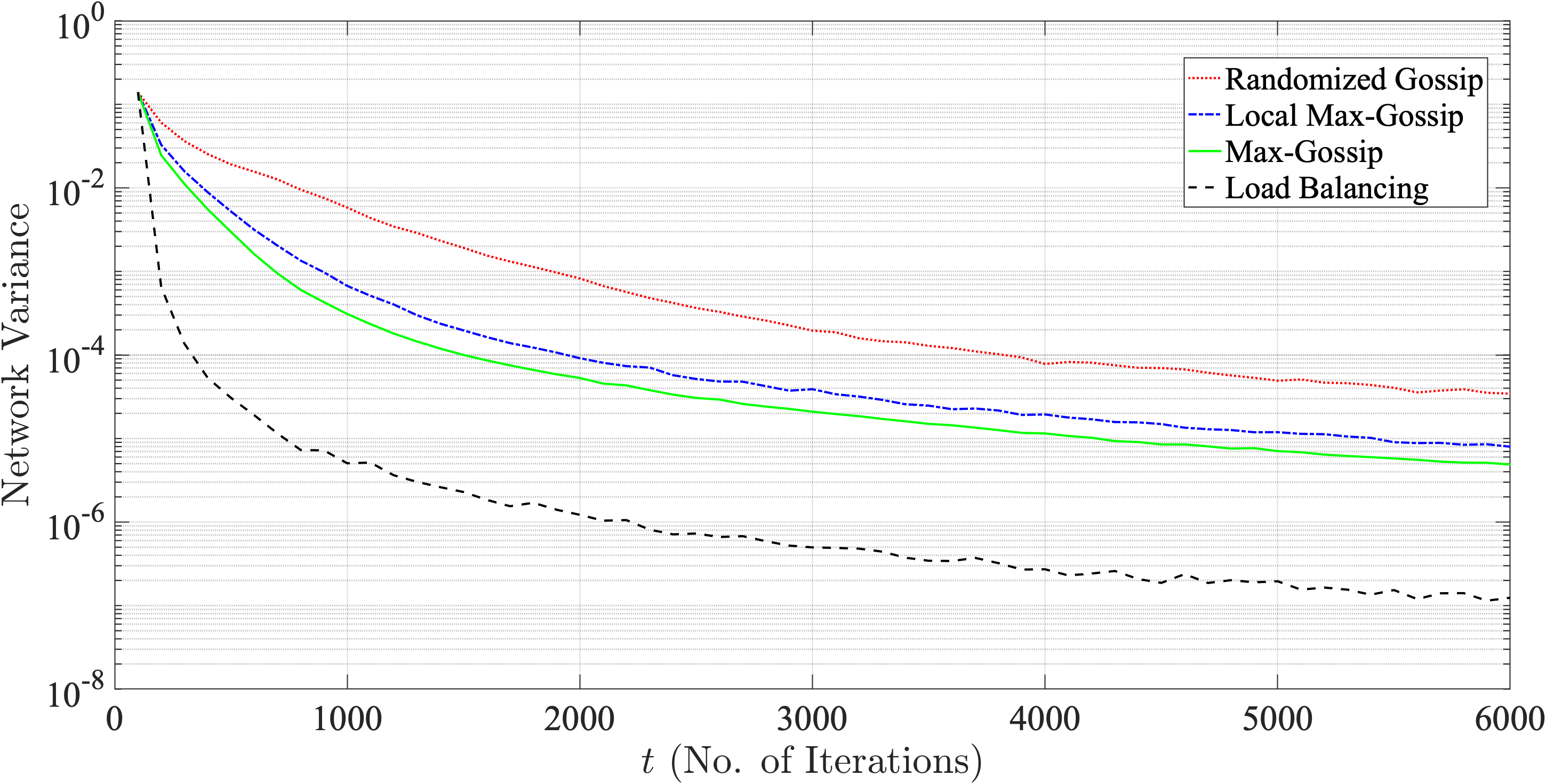}
\end{subfigure}
\begin{subfigure}{0.5\textwidth}
    \centering
    \includegraphics[width=\textwidth]{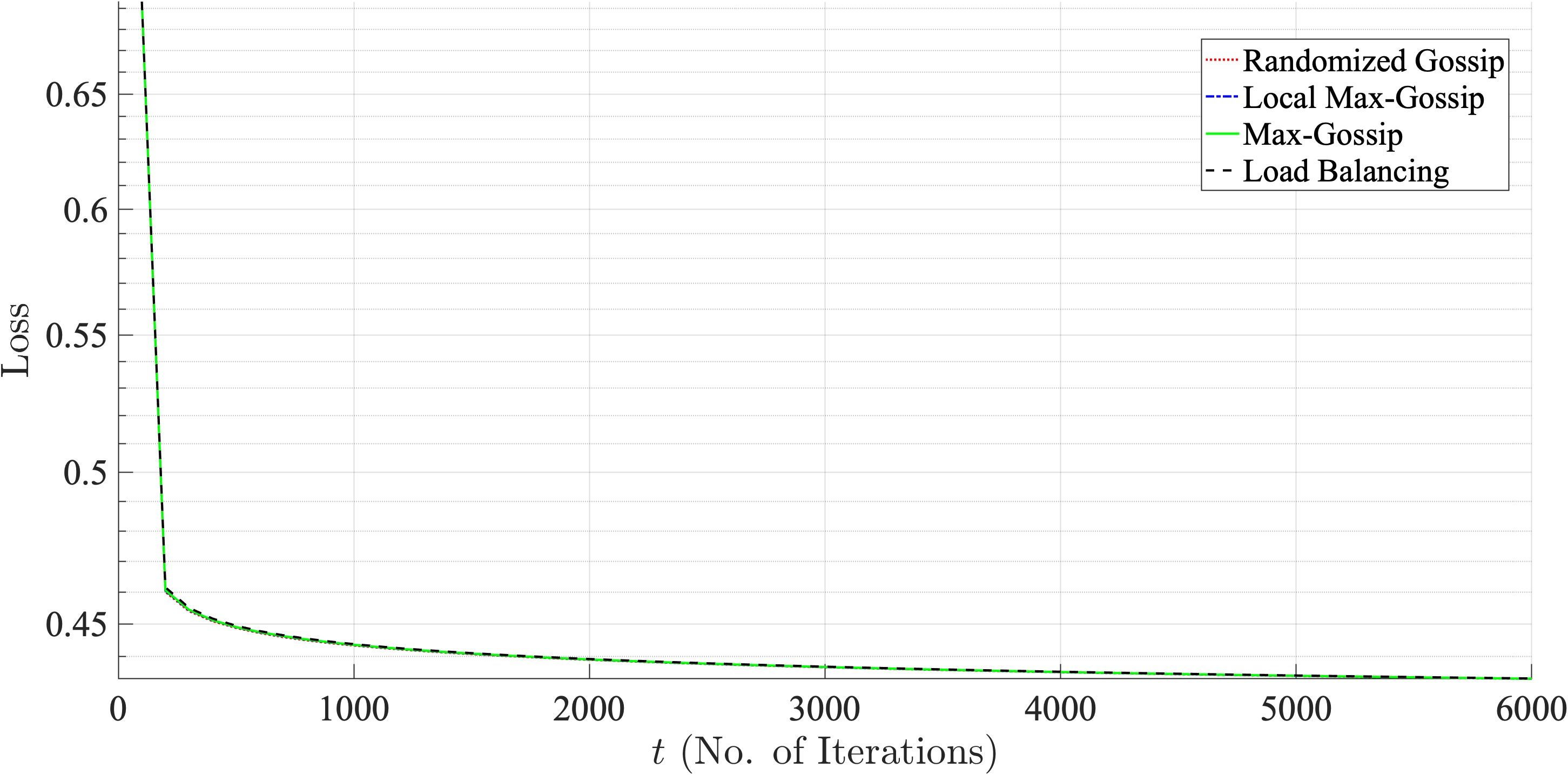}
\end{subfigure}
\caption{Network Variance for Ladder Graph with $20$ nodes}
\label{fig:network_variance}
\end{figure}

\color{black}

\section{Conclusions}\label{sec:conclude}
We proposed, studied, and analyzed the role of maximal dissent nodes in distributed optimization schemes, leading to many exciting state-dependent consensus-based subgradient methods. The proof of our result relies on a certain contraction property of these schemes. Our result opens up avenues for synthesizing or extending the use of state-dependent averaging-schemes for distributed optimization including the Max-Gossip, Local Max-Gossip, and Load-Balancing   algorithms.
Finally, we compared simulation results of a distributed estimation problem for gossip-based subgradient methods and the proposed state-dependent algorithms. Our numerical experiments show the faster convergence speed of schemes that use maximal dissent between nodes compared with state-independent gossip schemes. These simulations strongly support the intuition behind our main result, i.e., mixing of information between the maximal dissent nodes is critically important for the working (and enhancing) of the consensus-based subgradient methods. Although, we have shown the convergence of such state-dependent algorithms, establishing their rate of convergence, and especially relating them to various graph quantities such as diameter and edge density of the graph remains open problems for future research endeavors. {\color{black} The introduction of a state-dependent element for other class of algorithms specifically those which provide linear convergence rates such as distributed gradient tracking method \cite{qu2017harnessing,pu2021distributed} and their convergence analysis are part of future direction for the problem.}


\appendix

\section{Proof of \cref{prop:LBNecessity}} \label{appendix:LBNecessity}
\begin{proof}[Proof of \cref{prop:LBNecessity}]
For any $\omega\in \Omega$, consider $X(t;\omega) \in \R^{n\times d}$. If nodes $i$ and $j$ update their values to their average, that is $\big(\vecx_i(t;\omega) + \vecx_j(t;\omega)\big)/2$, then we know that during the round of Load-Balancing   algorithm starting at value $X(t;\omega)$ in step 2, node $i$ and node $j$ have sent their averaging request to each other. Therefore, we have $j \in \arg\max_{r \in \Ncal_i}\|\vecx_j(t;\omega) - \vecx_r(t;\omega)\|$ and $i \in \arg\max_{r \in \Ncal_j}\|\vecx_i(t;\omega) - \vecx_r(t;\omega)\|$. Hence, for any $\omega \in \Omega$,
\begin{equation}\label{eq:proofLBnecessity}\| \vecx_i(t;\omega) - \vecx_j(t;\omega) \| \geq  \max\Big\{\max_{r\in \Ncal_i\setminus\{j\}}\|\vecx_i(t;\omega) - \vecx_r(t;\omega)\|, \max_{r\in \Ncal_j\setminus\{i\}}\|\vecx_j(t;\omega) - \vecx_r(t;\omega)\|\Big\}.
\end{equation}

On the other hand, if \cref{eq:proofLBnecessity} holds  with strict inequality, then node $i$ and node $j$ send averaging requests only to each other in step 2 and respond to each other in step 3, and carry out their averaging according to step 4.
\end{proof}

\section{Proof of \cref{prop:LMLBContracting}}\label{appendix:LMLBContracting}

\begin{proof}
We first discuss the result for Randomized Gossip, Local Max-Gossip, and Max-Gossip averaging. The averaging matrices for the gossip algorithms where two agents update their states to their average takes the form of  \cref{eq:matrix_form_1}. Therefore, for these gossip algorithms we have 
\begin{equation}
    A\big(t,X(t)\big)^T \! \! A\big(t,X(t)\big) = A\big(t,X(t)\big) \nonumber
\end{equation}
and 
\begin{equation}
    \E\Big[A\big(t,X(t)\big)^T A\big(t,X(t)\big) \mid \Fcal_t \Big] = \E\Big[A\big(t,X(t)\big) \mid \Fcal_t \Big]. \nonumber
\end{equation}

Consider two nodes $ i,j \in \Vcal $ such that  $\{i,j\}\in \Ecal$. For Randomized Gossip, $\E\big[A\big(t,X(t)\big)_{ij} \mid \Fcal_t\big] = (P_{ij} +P_{ji})/2n.$ Moreover, since $\{i,j\} \in \Ecal$, we have $P_{ij},P_{ji}>0$. 
Let $P_* = \min_{\{i,j\}\in \Ecal} P_{ij}$. For the max-edge $\{i^*,j^*\},$ \cref{eqn:maxEdgeDelta} holds with ${\delta = P_*/n >0}$.
 
Let $i\in \Vcal$ and state estimate matrix $X(t)$. 
For Local Max-Gossip, let  $r_i\big(X(t)\big)$ be determined according to \cref{eqn:rsxt}. 
Consider the max-edge $\emax\big(X(t)\big)=\{i^*,j^*\}$. Then, 
$r_{i*}\big(X(t)\big) = j^*$ and $r_{j^*}\big(X(t)\big)=i^*$. Thus,
\begin{align}
    \E\Big[A\big(t,X(t)\big)_{i^*j^*} \mid \Fcal_t\Big]  =\frac{1}{n} \nonumber
\end{align}
and Local Max-Gossip averaging satisfies inequality \cref{eqn:maxEdgeDelta} with $\delta = 1/n.$

Similarly, for the Max-Gossip averaging with state estimate $X(t)$ at time $t$, for the max-edge $\emax\big(X(t)\big)=\{i^*,j^*\}$, we have
\begin{equation}
    \E\left[A\big(t,X(t)\big)_{i^*j^*}  \mid \Fcal_t \right] = \frac{1}{2}, 
    \nonumber
\end{equation}
and \cref{eqn:maxEdgeDelta}  holds with $\delta = 1/2$.

Let us now discuss the presence of max-edge in the Load-Balancing   averaging scheme. Consider the state estimate matrix $X(t)$  and
$\emax\big(X(t)\big) = \{i^*,j^*\}$ to be the max-edge with respect to $X(t)$. By the definition of a max-edge we know that nodes $i^*, j^*$ satisfy inequality \cref{eq:LBNecessity}. 

Consider the case when nodes $i^*,j^*$ satisfy \cref{eq:LBNecessity} with strict inequality. From \cref{prop:LBNecessity}, we know that $A(t,X(t))_{i^*j^*}, A\big(t,X(t)\big)_{j^*i^*}, A\big(t,X(t)\big)_{i^*i^*},  A\big(t,X(t)\big)_{j^*i^*}$ are equal to $1/2,$ which implies that $A\big(t,X(t)\big)_{i^*\ell}=A\big(t,X(t)\big)_{\ell j^*}=0$ for all $\ell \not \in \{i^*,j^*\}$. Therefore, \begin{equation}
    \E\Big[A\big(t,X(t)\big)^T \!\! A\big(t,X(t)\big) \mid \Fcal_t\Big]_{i^*j^*} = 1/2, \nonumber
\end{equation} and the inequality in \cref{eqn:maxEdgeDelta} holds with $\delta = 1/2$. 

Finally, consider the case when there are multiple neighbors of nodes $i^*, \,j^*$ with distance equal to $\|\vecx_{i^*}(t)-\vecx_{j^*}(t)\|$. Let $|\Scal_{i^*}|\geq 1$  and $|\Scal_{j^*}|\geq 1$ where  $\Scal_i$ is given by \cref{eq:max_dissent_set}.
Then, according to Load-Balancing   algorithm, nodes $i^*, \, j^*$ update their states to their average with probability $1/(|\Scal_{i^*}| \cdot |\Scal_{j^*}|)$. Since $|\Scal_{i^*}|\leq n-1$ and $|\Scal_{j^*}|\leq n-1$, we have 
\begin{equation}
    \E \Big[A\big(t,X(t)\big)^T \! \! A\big(t,X(t)\big)_{i^*j^*} \mid \Fcal_t \Big] \geq \frac{1}{2(n-1)^2}, \nonumber
\end{equation}
and \cref{eqn:maxEdgeDelta} holds with $\delta = 1/2(n-1)^2.$
\end{proof}

\section{Proof of Theorem~\ref{thm:maxEdgeContraction}}\label{appendix:edgecontraction}
To prove  \cref{thm:maxEdgeContraction} we must first define a few quantities related to the distance between the nodes on the graph and their relationships.
\begin{definition}\label{def:distances}
Consider a connected graph $\Gcal$ and a matrix $X = [\vecx_1,\dots,\vecx_n]^T \in \R^{n\times d}$ such that $\vecx_i \in \R^d$ is the 
estimate at node $i$ in the graph $\Gcal$. Let $d(X)$ denote the \underline{maximal distance} between the estimates of \textit{any} two nodes in the graph
\begin{equation}\label{eq:dist_any}
    d(X) \triangleq \max_{i,j \in \{1,2,\dots,n\}} \|\vecx_i - \vecx_j\|.
\end{equation}
Let $d_{\Gcal}(X)$ denote the maximal distance between the estimates among any two \textit{connected} nodes in the graph
\begin{equation}\label{eq:dist_edge}
    d_{\Gcal}(X) \triangleq \max_{\{i,j\} \in \Ecal} \|\vecx_i - \vecx_j\|.
\end{equation}
Finally, let  $\diam$ denote the \textit{longest shortest path} between any two nodes of the graph $\Gcal$.
\end{definition}

\begin{proposition}\label{prop:rel_d}
Given a connected graph $\Gcal$ and a matrix $X = [\vecx_1,\dots,\vecx_n]^T \in \R^{n\times d}$, such that $\vecx_i \in \R^d$ is the solution estimate at node $i$ in the graph $\Gcal$, we have 
\begin{equation}
    \frac{d(X)}{\diam} \leq d_{\Gcal}(X) \leq d(X). \nonumber
\end{equation}
\end{proposition}

\vspace{5pt}

\begin{proof}
The upper bound on $d_{\Gcal}(X)$ follows  from \cref{eq:dist_any,eq:dist_edge} in 
\cref{def:distances}.
To prove the lower bound on $d_{\Gcal}(X)$, we assume, without loss of generality, that the rows of the matrix $X \in \R^{n\times d}$ are such that 
$d(X) = \|\vecx_1 - \vecx_n\|$. 
Since $\mathcal{G}$ is connected, its diameter is finite and there is a path of length $k\leq \diam$, denoted by $\{v_0,v_1\},\{v_1,v_2\},\dots,\{v_{k-1},v_k\}$, where $v_0=1$ and $v_k=n$, with $v_i\in \Vcal$ for $i = 0,1,\ldots,k$. The distance $d(\vec{x})$ is bounded as 
\begin{equation}
\label{eq:triangle_ineq}
    \|\vecx_1 - \vecx_n\| 
       \leq  \sum_{i=0}^{k-1} \|\vecx_{v_i} - \vecx_{v_{i+1}} \|, 
\end{equation}
where Eq.~\eqref{eq:triangle_ineq} follows from the triangle inequality. Finally, each term in the sum Eq.~\eqref{eq:triangle_ineq} is bounded above by $d_{\Gcal}(\vec{x})$. Hence, \begin{equation}
  d(X) \leq kd_{\Gcal}(X)\leq \diam d_{\Gcal}(X).  \nonumber
\end{equation}
\end{proof}

Next, we state a result quantifying the decrease in the Lyapunov function defined in \cref{eq:Lyapunov} that is the vector form of~\cite[Lemma~1]{nedic2009distributedquantization}. \begin{lemma}\label{lemma:DecreaseLyapunov}
Given a doubly stochastic matrix $A \in \R^{n\times n}$, let $c_{ij}$ denote the $(i,j)$-th entry of the matrix $A^TA$. Then for all $X=[\vecx_1,\dots,\vecx_n]^T \in \R^{n\times d}$, we have 
\begin{equation}
V(AX) = V(X) - \sum_{i<j} c_{ij} \|\vecx_i - \vecx_j\|^2. \nonumber
\end{equation}

\end{lemma}
\begin{proof}
By definition, the Lyapunov function in \cref{eq:Lyapunov} can be written as
\begin{equation}
    V(X) = \mathrm{tr}\big[(X-\bar{X})^T(X-\bar{X})\big],\nonumber
\end{equation}
where $\bar{X} = \frac{\one\one^T}{n}X$. The doubly stochasticity of $A$ implies
$\overline{AX}= \frac{\one \one^T}{n} AX = \frac{\one\one^T}{n}X = \frac{A\one\one^T}{n}X = A\bar{X}$. Therefore, 
\begin{align}
    V(AX) & = & \Tr[(AX- A\bar{X})^T(AX- A\bar{X})]. \nonumber
\end{align}
Finally,     
\begin{equation}
    V(X)-V(AX) = \Tr[(X- \bar{X})^T (I- A^TA)(X- \bar{X})]. \nonumber
\end{equation}

Since $A^TA$ is a symmetric and stochastic matrix, we have $c_{ij}= c_{ji}$ and $c_{ii} =1 - \sum_{i\neq j}c_{ij}$. Thus, 
\begin{equation}
    A^T A = I - \sum_{i<j}c_{ij}(\vec{b}_i - \vec{b}_j)(\vec{b}_i - \vec{b}_j)^T,\nonumber 
\end{equation}
where  $\vec{b}_i\in \mathbb{R}^n$  is the standard basis vector for all $i \in \Vcal$. 
Since 
\begin{equation}
\Tr[(X- \bar{X})^T(\vec{b}_i - \vec{b}_j)(\vec{b}_i - \vec{b}_j)^T(X-\bar{X})] = \|\vecx_i - \vecx_j\|^2, \nonumber
\end{equation}
we have 
\begin{align}
    V(X)- V(AX) = \sum_{i<j}c_{ij}\|\vecx_i - \vecx_j\|^2. \nonumber
\end{align}
\end{proof}

\begin{proof}[Proof of Theorem~\ref{thm:maxEdgeContraction}]
At time $t\geq 0$ consider the state estimate $X(t) = [\vecx_1(t), \dots, \vecx_n(t) ]^T \in \R^n$, the corresponding max-edge $\emax(X(t))=\{i^*,j^*\}$ and the doubly stochastic averaging matrix $A(t,X(t))$ such that
\begin{equation}
    \E[A(t,X(t))^T\!\! A(t,X(t))_{i^*j^*} \mid \Fcal_t] \geq \delta >0 \ \ \text{a.s.}  \nonumber
\end{equation}
Define $$\Omega_{\delta}(t) = \{\omega: \E[A(t,X(t))^T\!\! A(t,X(t))_{i^*j^*} \mid \Fcal_t] \geq \delta\}.$$ 
For legibility, we drop the time index in the variables for the rest of this proof and use $X$, $\Fcal$, $A(X)$, $\Omega_{\delta}$ instead of $X(t), \Fcal_t,  \Omega_{\delta}(t)$, and $A(t,X(t))$.

Using arguments similar to the ones from~\cite[Lemma~9]{nedic2009distributedaveraging}, for $X \in \Fcal$ and doubly stochastic matrix $A(X)$ such that
\begin{equation}
    \mathbb{E}[\big(A(X)^T \! \! A(X)\big)_{i^*j^*} \mid \Fcal]\geq \delta>0 \ \ \text{a.s.}, \label{eqn:AssumptionDecreaseInMaxEdge}
\end{equation} 
where $\Fcal$ is a $\sigma$-field, $X\in \Fcal$, and ${\emax(X) = \{i^*,j^*\}}$. We will show that  ${\mathbb{E}\big[V(A(X)X)\mid \Fcal\big] \leq \lambda V(X)}$ a.s. for some $\lambda \in (0,1)$. From \cref{lemma:DecreaseLyapunov}, the difference in the quadratic Lyapunov function $V$ evaluated at $X$ and $A(X)X$ is given by
\begin{equation}
    V(X) - V(A(X)X) = \sum_{i<j} c_{ij}(X)\|\vecx_i - \vecx_j\|^2, \nonumber
\end{equation}
where $c_{ij}(X)$ is the  $(i,j)$-th entry of $ A(X)^TA(X) $, {i.e.},  $c_{ij}(X) = (A(X)^T \! \! A(X))_{ij}$.
Taking the conditional expectation with respect to the filtration $\mathcal{F}$, we obtain 
\begin{align}
    V(X) &- \E\big[V(A(X)X) \mid \Fcal\big] \nonumber\\
    & =  \sum_{i<j}\big(\E[\big(A(X)^T \!\! A(X)\big)_{ij} \mid \Fcal]\big)\|\vecx_i - \vecx_j\|^2  \nonumber\\
    & \geq  c_{i^*j^*}(X) \|\vecx_{i^*} - \vecx_{j^*}\|^2  \geq \delta  \|\vecx_{i^*} - \vecx_{j^*}\|^2 \ \text{a.s.}, \nonumber
\end{align}
where  $\emax(X)= \{i^*,j^*\}$ and the first inequality follows from the non-negativity of the squared terms and the second inequality follows from \cref{eqn:AssumptionDecreaseInMaxEdge}.
Recall that the constant $\delta$ depends on the averaging scheme. 

If $V(X) = 0$, more precisely for the samples path characterized by $\omega\in \Omega_{\delta}(t)$ such that $V(X(t;\omega))=0$\footnote{We omit the dependency on $\omega$ and $t$ for legibility.}, then $X =\one \vec{c}^T$ for some $\vec{c} \in \R^{d}$. Therefore, $A(X)X = A(X)\one\vec{c}^T=\one\vec{c}^T$ since $A(X)$ is doubly stochastic and $V(A(X)X) = 0$. Thus, the inequality $\mathbb{E}\big[V(A(X)X)\mid \Fcal\big] \leq \lambda V(X)$ is satisfied. 

Let ${\Lcal \triangleq \{ \one \vec{p}^T \mid \vec{p}\in \R^d\}}$. For $X= [\vecx_1,\cdots,\vecx_n]^T \not\in \Lcal $, more precisely for the samples path characterized by $\omega\in \Omega_{\delta}(t)$ such that $X(t;\omega)\not\in\mathcal{L}$, the conditional expected fractional decrease in the Lyapunov function is 
\begin{equation}
    \frac{V(X)-\E[V(A(X)X) \mid \Fcal]}{V(X)} \geq \delta \frac{\|\vecx_{i^*} - \vecx_{j^*}\|^2}{\sum_{i=1}^n \|\vecx_i - \bar{\vecx}\|^2},  \nonumber
\end{equation}
where $\bar{\vecx} = \frac{1}{n}\sum_{i=1}^n \vecx_i$. Using the definition of $d_{\Gcal}(X)$ and Proposition~\ref{prop:rel_d}, we obtain the following bound
\begin{equation}
     \frac{V(X)-\E[V(AX)\mid \Fcal]}{V(X)} \geq \frac{\delta}{\diam^2}\frac{d^2(X)}{\sum_{i=1}^n \|\vecx_i - \bar{\vecx}\|^2} \label{eq:Vt_ratio_1}. \nonumber
\end{equation}

For $X\not\in \mathcal{L}$, let
\begin{equation}
    g(X)\triangleq \frac{d^2(X)}{\sum_{i=1}^n \|\vecx_i - \bar{\vecx}\|^2}.\nonumber
\end{equation} 
Note that $g(X)$ satisfies the following invariance relations 
\begin{equation}
    g(X + \one\vec{p}^T)=g(X), \ \ \vec{p} \in \R^d,  \nonumber
\end{equation}
 and 
 \begin{equation}
      g(cX)=g(X), \ \ c \in \R\setminus\{0\}. \nonumber
 \end{equation}
Therefore, for $X\not\in \Lcal$ the following inequality and identity hold
\begin{align}
g(X)& \geq & \min_{Z \in \mathbb{R}^{n\times d} : \sum_i \vec{z}_i =
    \vec{0}}\frac{d^2(Z)}{\sum_{i=1}^n \|\vec{z}_i\|^2} \nonumber \\
   &=& \min_{Z \in \mathbb{R}^{n\times d} :  \sum_i \vec{z}_i = \vec{0}; \sum_{i} \|\vec{z}_i\|^2 =1}{d^2(Z)}. \nonumber
\end{align}
Note that if $\sum_{i=1}^n \vec{z}_i=\vec{0}$ and $\sum_{i=1}^n \|\vec{z}_i\|^2=1$, then we have
\begin{equation}
    \sum_{1\leq i <j \leq n} \inp{\vec{z}_i}{\vec{z}_j} 
    =  - \frac{1}{2}\sum_{i=1}^n \|\vec{z}_i\|^2  = -\frac{1}{2}. \label{eq:inpValue} 
\end{equation}
By definition, $d(Z) \geq \|\vec{z}_i - \vec{z}_j \|$ for all $i,j \in [n]$. Using the fact that the maximum of a set of values is greater than its average for the set $\{\|\vec{z}_{i} -\vec{z}_j\|^2\}_{1\leq i < j \leq n}$, we get
\begin{align}
    d^2(Z) &\geq& \frac{2}{n(n-1)} \sum_{1 \leq i < j \leq n} \|\vec{z}_i - \vec{z}_j\|^2  =
     \frac{2}{n-1}, \nonumber
\end{align}
where the last step follows from \cref{eq:inpValue} and the fact that $\sum_{i=1}^n \|\vec{z}_i\|^2 = 1$. Finally, using  \cref{eq:Vt_ratio_1}, we get 
\begin{equation}
     \frac{V(X)-\E\big[V(A(X)X) \mid \Fcal\big]}{V(X)} \geq \frac{2\delta}{(n-1) \diam^2}.  \nonumber
\end{equation}
Since $\mathbb{E}\big[V(A(X)X)\mid \Fcal\big] \leq \lambda V(X)$ for $X \in \Lcal$ and for $X \not\in \Lcal$, we have $\mathbb{E}\big[V(A(X)X)\mid \Fcal\big] \leq \lambda V(X)$ a.s.
Thus,
\begin{equation}
    \E\big[V(A(t,X(t))X(t)) \mid \Fcal_t\big] \leq \lambda V\big(X(t)\big) \ \ \text{a.s.}, \nonumber
\end{equation}
where $\lambda =1- 2\delta/\big((n-1)\diam^2\big).$
\end{proof}
\section{Limiting properties of the Lyapunov function $V(\cdot)$}\label{appendix:distfrommean}
To prove \cref{lemma:dist_from_mean} we will make use of the following result.

\begin{theorem}[{Robbins-Siegmund Theorem}]\label{thm:R-S}
Let $(\Omega,\Fcal,\mathcal{P})$ be a probability space and ${\Fcal_0\subseteq \Fcal_1 \subseteq \cdots}$ be a sequence of sub $\sigma$-fields of $\Fcal$. Let $\{u_t\},\{v_t\},\{q_t\}$, and $\{w_t\}$ be $\Fcal_t$-measurable random variables, where $\{u_t\}$ is uniformly bounded from below, and $\{v_t\}$, $\{q_t\}$, and $\{w_t\}$ are non-negative. Let $\sum_{t=0}^{\infty} w_t<\infty,\  \sum_{t=0}^{\infty} q_t <\infty $ and 
\begin{align}
    \E [u_{t+1} \mid \Fcal_t] \leq (1+q_t)u_t -v_t +w_t, \ \ \text{a.s.,} \nonumber
\end{align}
for all $t\geq 0$. Then, the sequence $\{u_t\}$ converges and $\sum_{t=0}^{\infty} {v_t} < \infty$ a.s.
\end{theorem}

\begin{proof}[Proof of Lemma~\ref{lemma:dist_from_mean}] 
To study the convergence of $V\big(W(t)\big)$, we first derive a super-martingale like inequality for the stochastic process $\big\{V\big(W(t)\big)\big\}$.
For $X(t)\in \Fcal_t$ using the contracting averaging property  of $A(t,X(t))$ in \cref{eq:defContractingMatrix},  we get
\begin{align}
    \E\big[V\big(W(t+1)\big)\mid \Fcal_t\big] &= \E\big[V\big(A(t,X(t))X(t)\big) \mid \Fcal_t\big] \cr 
    &\leq  \lambda V\big(X(t)\big), \ \ \text{a.s.}, \label{eqn:proofContraction}
\end{align}
where $\lambda \in (0,1)$. We know that $X(t) = W(t) + \e(t)$, so from triangle inequality on $\|W(t)-\bar{W}(t)+E(t)-\bar{E}(t)\|_F$ we have
\begin{equation}\label{eq:triangle_ineq_VX}
    V\big(X(t)\big) \leq V\big(W(t)\big) + V\big(\e(t)\big) + 2\sqrt{V\big(W(t)\big)}\sqrt{V\big(\e(t)\big)}. 
\end{equation}
    
Using the inequality above in \cref{eqn:proofContraction}, for all $t\geq 0$ we get  
\begin{equation}
    \E\big[V\big(W(t+1)\big)\mid \Fcal_t \big] \leq \lambda \bigg(V\big(W(t)\big) + V\big(\e(t)\big)  + 2\sqrt{V\big(W(t)\big)}\sqrt{V\big(\e(t)\big)}\bigg) \ \ \text{a.s.} \nonumber
\end{equation}
Since $V\big(\e(t)\big) = \|\e(t)-\bar{\e}(t)\|_F^2 \leq \|\e(t)\|_F^2 \leq L^2\alpha^2(t)$, we get 
\begin{align}
     \E[V\big(W(t+1)\big)\mid \Fcal_t] &\leq \lambda \left(\sqrt{V\big(W(t)\big)} + L\alpha(t)\right)^2 \ \ \text{a.s.}\nonumber
\end{align}
From Jensen's inequality, we have
\begin{equation}
    \E\Big[\sqrt{V\big(W(t+1)\big)} \mid \Fcal_t \Big] \leq  \sqrt{\E\big[V\big(W(t+1)\big) \mid \Fcal_t\big]}
     \leq  \sqrt{\lambda}\Big(\sqrt{V\big(W(t)\big)} + L\alpha(t)\Big) \ \ \text{a.s.} \nonumber
\end{equation}

Taking the expectation, multiplying by $\alpha(t+1)$ and using the fact that $\{\alpha(t)\}$ is non-increasing, we get
\begin{equation}
    \alpha(t+1) \E \big[\sqrt{V\big(W(t+1)\big)}\big] 
    \leq\alpha(t)\E\sqrt{V\big(W(t)\big)} 
      - (1-\sqrt{\lambda})\alpha(t)\E\sqrt{V\big(W(t)\big)} + \alpha^2(t) \ \ \text{a.s.} \nonumber
\end{equation}
Since the diminishing step sequence $\{\alpha(t)\}$ satisfies $\sum_{t=1}^{\infty}\alpha^2(t)<\infty$, \cref{thm:R-S} results in
\begin{equation}
    \sum_{t=1}^{\infty}\alpha(t)\E\sqrt{V\big(W(t)\big)}<\infty, \ \  \nonumber
\end{equation}
and by the Monotone Convergence Theorem, we have,
\begin{equation}\label{eq:MCTapplied}
    \E\left[\sum_{t=1}^{\infty}\alpha(t)\ \sqrt{V\big(W(t)\big)}\right]<\infty,
\end{equation}
which implies that
\begin{equation}
    \sum_{t=1}^{\infty}\alpha(t)\ \sqrt{V\big(W(t)\big)}<\infty, \ \text{a.s.} \nonumber
\end{equation}  
Since $V(W(t)) = \sum_{i=1}^n \|\vecw_i(t)-\bar{\vecw}(t)\|^2$, we know that 
\begin{align}
    \sum_{t=1}^{\infty} \alpha(t)\|\vecw_i(t) - \bar{\vecw}(t)\| \leq \sum_{t=1}^{\infty}\alpha(t)\sqrt{V\big(W(t)\big)}<\infty, \nonumber
\end{align}
for all $i\in [n]$, a.s.
Since $\sum_{t=1}^{\infty} \alpha(t)\|\vecw_i(t)-\bar{\vecw}(t)\| <\infty$ and $\sum_{t=1}^{\infty}\alpha(t)=\infty $, we have
\begin{equation}\label{eq:liminfw}
    \liminf_{t\to\infty} \|\vecw_i(t)-\bar{\vecw}(t)\| = 0, \ \forall i \in [n], \ \  \text{  a.s.}
    \end{equation}
Further since we have,
\begin{equation}
    \sum_{t=1}^{\infty}\alpha(t)\E\sqrt{V\big(W(t)\big)} = \E\left[\sum_{t=1}^{\infty}\alpha(t)\E\left[\sqrt{V\big(W(t)\big)} \mid \Fcal_t\right]\right], \nonumber
\end{equation}
using Monotone Convergence Theorem similar to \cref{eq:MCTapplied} implies that 
\begin{equation}
    \E\left[\sum_{t=1}^{\infty}\alpha(t)\E\left[ \|\vecw_i(t) - \bar{\vecw}(t) \| \mid \Fcal_t\right]\right]<\infty, \nonumber
\end{equation}  
and so, we  have
\begin{equation}
    \sum_{t=1}^{\infty}\alpha(t)\E\left[\sqrt{V\big(W(t)\big)} \mid \Fcal_t\right]<\infty \ \ \text{a.s.}, \nonumber
\end{equation}  
and therefore,
\begin{equation}\label{eq:final1}
    \sum_{t=1}^{\infty}\alpha(t)\E\left[ \|\vecw_i(t) - \bar{\vecw}(t) \| \mid \Fcal_t\right]<\infty, \ \forall i \in [n], \ \text{a.s.}
\end{equation}  
Further, for all $t\geq 0$, we know 
\begin{equation}
    \E\big[V\big(W(t+1)\big) \mid \Fcal_t\big]
    \leq  \lambda \left(V\big(W(t)\big) + 2L\alpha(t)\sqrt{V\big(W(t)\big)} +L^2\alpha^2(t)\right) \ \text{a.s.} \nonumber
\end{equation}
Since we have
\begin{equation}
    \sum_{t=1}^{\infty} 2\alpha(t)\sqrt{V\big(W(t)\big)} + \lambda L^2\alpha^2(t)<\infty \quad \textrm{a.s.}, \nonumber
\end{equation}
\Cref{thm:R-S} implies that $    \{V(W(t))\} \   \text{converges a.s.}$
 Therefore, 
 \begin{equation}
     \|\vecw_i(t+1) - \bar{\vecw}(t+1) \|  \ \text{converges,} \ \forall i \in [n],\ \ \textrm{a.s.} \nonumber
 \end{equation} 
Using \eqref{eq:liminfw} with the above result, we get
 \begin{equation}\label{eq:final2}
    \lim_{t\to\infty} \|\vecw_i(t+1) - \bar{\vecw}(t+1)\| = 0, \ \forall i\in[n],\ \ \text{a.s.}
 \end{equation}

Finally, since $\bar{\vecw}(t+1)^T = \frac{\one^TW(t+1)}{n} = \frac{\one^TA(t,X(t))X(t)}{n}$ from the double stochasticity of $A(t,X(t))$, we have 
\begin{equation}
    \bar{\vecw}(t+1)^T=\frac{\one^TX(t)}{n} = \bar{\vecx}(t)^T, \nonumber
\end{equation}
which from \cref{eq:final1,eq:final2} implies  \cref{lemma:dist_from_mean}.
\end{proof}

\section{Proof of Lemma~\ref{lemma:convex}} \label{appendix:lemmaConvex}
To prove Lemma~\ref{lemma:convex}, we follow the proof in~\cite[Theorem~1]{aghajan2020distributed}.
\begin{proof}
For all $\vecx \in \mathcal{X}^*$ and $t \geq 0$, we have
\begin{equation}
    \E\big[\|\vecx_{t+1} - \vecx \|^2 \mid \Fcal_t \big]  
    \leq (1+b_t)\|{\vecx}_t - \vecx\|^2 - a_t\big(f(\vecx_t) - f(\vecx)\big) +c_t \ \text{a.s.} \nonumber
\end{equation}

For any $\vecx  \in \mathcal{X}^*$, \Cref{thm:R-S} 
implies that
${\{\|\vecx_t-\vecx\|\} \ \text{converges}}$    
and
\begin{equation}
    \sum_{t=0}^{\infty} a_t\big(f({\vecx}_t)-f(\vecx)\big)<\infty \ \ \text{a.s.} \nonumber
\end{equation}
Since for any $\vecx\in \mathcal{X}^*$ we have $f(\vecx) = f^*$, the event
\begin{equation}
    \Omega_\vecx = \bigg\{
         \omega :   \lim_{t\to \infty} \|\vecx_t(\omega) - \vecx\| \textrm{ exists, and }   \sum_{t=0}^{\infty} a_{t}(f(\vecx_t(\omega))-f^*)<\infty   \bigg\} \nonumber
\end{equation}
is such that 
$\mathbb{P}(\Omega_\vecx) = 1$. 
Note that here we denote by $\{{\vecx}_t(\omega)\}_{t\geq 0}$ the sample path for the corresponding $\omega$.

Let $\mathcal{X}_d^* \subseteq \mathcal{X}^*$ be a countable dense subset of $\mathcal{X}^*$ and ${\Omega_d = \bigcap_{\vecx\in \mathcal{X}_d^*} \Omega_\vecx.}$ We have $\mathbb{P}(\Omega_d) = 1$ since $\mathcal{X}_d^*$ is countable. For any $\omega \in \Omega_d$, since $\sum_{t=0}^{\infty} a_t = \infty$ and $\sum_{t=0}^{\infty} a_t\big(f({\vecx}_t(\omega)-f^*\big)<\infty $, we have
    \begin{equation}\label{eqn:fliminf}
        \liminf_{t\to \infty} f\big({\vecx}_t(\omega)\big) = f^*.
    \end{equation}

From \cref{eqn:fliminf} and the continuity of $f$, for all $\omega \in \Omega_d$, we have
\begin{align}
    \liminf_{t \to \infty} \|\vecx_t(\omega) - \vec{x}^*(\omega)\|=0, \nonumber
\end{align}
for some $\vec{x}^*(\omega) \in \mathcal{X}^*$\footnote{$\vec{v}^*(\omega)$ may not be in $\mathcal{X}^*_d$.}. Consider a subsequence $\{\vecx_{t_k}(\omega)\}_{k\geq0}$ of $\{\vecx_t(\omega)\}_{t\geq 0}$ such that 
\begin{equation}
    \lim_{k \to \infty} f\big(\vecx_{t_k}(\omega)\big) = f^*. \nonumber 
\end{equation} 

For any $\omega \in \Omega_d$, $\lim_{t \to \infty} \|\vecx_t(\omega)-\hat{\vecx}\|$ exists for $\hat{\vecx} \in \mathcal{X}_d^*$. Therefore, the sequences $\{\vecx_t(\omega)\}_{t\geq0}$ are bounded.
Hence, $\{\vecx_{t_k}(\omega)\}_{k\geq 0}$ is also bounded, has a limit point $\vec{x}^*(\omega) \in \mathcal{X}^*$, and without loss of generality, 
\begin{equation}
    \lim_{k \to \infty} \vecx_{t_k}(\omega) = \vecx^*(\omega). \nonumber
\end{equation} 
Since $\mathcal{X}_d^*$ is dense, there is a sequence $\{\vec{q}_s(\omega)\}_{s\geq0}$ in $\mathcal{X}_d^*$ such that 
\begin{equation}
    \lim_{s\to \infty} \|\vec{q}_s(\omega)-\vec{x}^*(\omega)\| = 0. \nonumber
\end{equation}
    
For $\omega \in \Omega_d$, $\lim_{t\to \infty} \|\vecx_t(\omega) - \vec{q}_s(\omega)\|$ exists for all $s\geq 0$, which is $\|\vecx^*(\omega) - \vec{q}_s(\omega)\|$. 
Moreover,
\begin{equation}
    \lim_{t\to\infty} \|\vecx_t(\omega) - \vec{q}_s(\omega)\| 
    \leq  \liminf_{t\to\infty} \|\vecx_t(\omega) - \vecx^*(\omega)\|  + \| \vecx^*(\omega)- \vec{q}_s(\omega)\| \leq  \| \vecx^*(\omega)- \vec{q}_s(\omega)\|, \nonumber
\end{equation}
which implies that
\begin{equation}
    \lim_{s\to \infty}\lim_{t\to \infty} \|\vecx_t(\omega) - \vec{q}_s(\omega)\| = 0. \nonumber
\end{equation}
Finally,
\begin{equation}
    \limsup_{t \to \infty} \|\vecx_t(\omega) - \vecx^*(\omega)\| \leq \lim_{s \to \infty} \limsup_{t \to \infty} \|\vecx_t(\omega)- \vec{q}_s(\omega)\|  + \|\vec{q}_s(\omega)-  \vecx^*(\omega)\| = 0. \nonumber
\end{equation}
Therefore, for any $\omega \in \Omega_d$, we have
$
    \lim_{t\to \infty} \vecx_t(\omega) = \vecx^*(\omega), \ \ 
$
where $\vecx^*(\omega) \in \mathcal{X}^*.$ So we have, ${\lim_{t\to \infty} \vecx_t= \vecx^*}$ a.s.
\end{proof}

\bibliographystyle{ieeetr}
\bibliography{bib}
\end{document}

%% file: global-local.tex
\subsection{Global Max-Gossip}
The standard gossiping algorithm described above is state-independent in the sense that the selection of the \textit{gossiping edge} $e$ does not depend on the states at the agents at any time.  Herein, we propose \textit{Global Max-Gossip} where we select
the edge connecting the agents with the largest possible  \textit{dissent} (disagreement) among all edges in the graph $\Gcal = (\Vcal, \Ecal)$, i.e.,
\begin{equation}
    \emax(\Gcal, X) =  \arg \max_{\{i,j\}\in\Ecal}\|\vecx_i-\vecx_j\|. \label{eq:emax}
\end{equation}
In case there are multiple solutions to \cref{eq:emax}, we select the smallest pair of indices $(i^*,j^*)$ based on the lexicographical order, without loss of optimality. For brevity, we use $\emax(X)$ to denote the \textit{max-edge}. 



\textit{Global Max-Gossip} serves as a benchmark as to what is achievable via state-dependent averaging schemes.  Global Max-Gossip requires an oracle to provide the edge resulting in the largest possible Lyapunov function reduction across all network edges.  
Obtaining a decentralized algorithm to determine the max-dissent edge is a challenging open problem beyond the scope of this paper.



Given an initial state matrix $X(0)$, the Max-Gossip averaging scheme admits a state-dependent dynamics of the form 
\begin{equation}
    A\big(t,X(t)\big) = B\Big(\emax\big(X(t)\big)\Big), \nonumber
\end{equation}
where the gossiping matrix is given by \cref{eq:matrix_form_1} and the max-edge is selected according to \cref{eq:emax}.  

\subsection{Local Max-Gossip}\label{sec:LMG}

In \textit{Local Max-Gossip} introduced in~\cite{ustebay2010greedy} under the moniker of \textit{Greedy Gossip with Eavesdropping}, a random selected node gossips with the  neighbor $j\in \mathcal{N}_{i}$ that has the largest\footnote{In case there are multiple solutions to \cref{eq:max}, we may select the agent with the smallest index, without loss of optimality.} possible state discrepancy with $i$, i.e.,
\begin{equation}\label{eq:max}
    j = \arg \max_{j\in \mathcal{N}_{i}} \|\vecx_j(t) - \vecx_{i}(t) \|.
\end{equation}
Convergence is accelerated by 
 gossiping with the neighbor with the largest disagreement as this leads to the largest possible immediate reduction in the Lyapunov function used to capture the variance of the states in the network. 

Since the edge over which the gossiping occurs depends on the current state of the neighbors, the resulting averaging matrix is a state-dependent, random matrix. For a sequence of independently and uniformly distributed index sequence $\{s(t)\}$, the Local Max-Gossip dynamics can be written as a state-dependent averaging scheme as follows 
\begin{equation}
    A\big(t,X(t)\big)=\Agossip\Big(\big\{s(t),r_{s(t)}\big(X(t)\big)\big\}\Big),\nonumber
\end{equation}
where 
\begin{equation}\label{eqn:rsxt}
    r_s(X) = \arg \max_{r\in \Ncal_s}\|\vecx_s-\vecx_r\|.
\end{equation}

%% file: convergence_rate.tex

\color{black}
\section{Convergence Rate}

In this section we discuss the convergence rate of the time-averaged version of the discussed state-dependent consensus based subgradient methods when the step size at time $t$ is set as $1/\sqrt{t}$ for $t\geq 1$. The convergence rates for the different algorithm differ via the contraction factor $\lambda$ defined for the contracting averaging matrix through \eqref{eq:defContractingMatrix}.  

Let $\lambda_t$ be the contraction factor defined through the contracting property of the matrices at time $t$. More precisely, for all $t \geq 0$ 
\begin{align*}
    \E [ V(A(t,X(t))X(t)) ] \leq \lambda_t V(X(t)),
\end{align*}
where $\lambda_t = \lambda \phi_t$ with $\phi_t \in (0,1)$. Here, $\lambda$ is the uniform bound on the contraction factor and $\lambda_t = \lambda_t(X(t))$ is a state-dependent (and possibly time-dependent) contraction factor. We refer to the tighter contraction bound to point out the improvement in convergence rate in state-dependent consensus based subgradient method. The proof of the convergence rates closely follow the proof provided in \cite{Nedic_2013}.

In the following lemma, we establish the convergence rate of the accumulation of error between the estimate for each agent from the mean of the estimates over all agents.  
\begin{lemma}\label{lemma:convergence_rate_x}
Under the assumptions of \Cref{thm:main1} with $\alpha(t) = 1/\sqrt{t}$, we have 
\begin{align}
    \frac{1}{\sqrt{n}}\sum_{k=0}^t \alpha(k+1) \sum_{i=1}^n\E[\|\vecw_i(k+1)- \bar{\vecw}_i(k+1)\|] \leq (K_1\E[\|X(0)-\bar{X}(0)\|_F] + LK_2(1+\ln t))
\end{align}
and 
\begin{align}\label{ineq:lemma_recursive2}
    \frac{1}{\sum_{k=0}^t \alpha(k+1)}&\sum_{k=0}^t \alpha(k+1) \sum_{i=1}^n \frac{1}{\sqrt{n}}\E[\|\vecw_i(k+1)- \bar{\vecw}_i(k+1)\|] \cr &\leq \frac{1}{\sqrt{t+1}}(K_1\E[\|X(0)-\bar{X}(0)\|_F] + LK_2(1+\ln t)),
\end{align}
where $K_1 = K_2 = \frac{\rootl}{1-\rootl}.$
\end{lemma}
\begin{proof}
From triangle inequality similar to \eqref{eq:triangle_ineq_VX}, we know for all $t \geq 1$ 
\begin{align}
    \E[\|W(t+1)-\bar{W}(t+1)\|_F] &\leq  \sqrt{\lambda_t} \E[\|W(t)-\bar{W}(t)\|_F] + \sqrt{\lambda_t} \E[\|E(t)-\bar{E}(t)\|_F]. \nonumber
\end{align}
Repeatedly applying the above inequality and since the perturbation is bounded as $V(E(t)) \leq \frac{L^2}{t}$ for all $t \geq 1$ we get 
\begin{align}
    \E[\|W(t+1)-\bar{W}(t+1)\|_F] &\leq  \prod_{s=1}^{t}\sqrt{\lambda_s} \E[\|W(1)-\bar{W}(1)\|_F] + \sum_{s=1}^t \prod_{k=s}^{t}\sqrt{\lambda_k} \E[\|E(s)-\bar{E}(s)\|_F] \cr 
    &\leq \prod_{s=0}^{t}\sqrt{\lambda_s} \E[\|W(0)-\bar{W}(0)\|_F] + \sum_{s=1}^t \prod_{k=s}^{t}\sqrt{\lambda_k} \frac{L}{\sqrt{s}}.\nonumber
\end{align}
For brevity, define $\phi(t:s)= \prod_{k=s}^{t}\phi(k)$ and rewrite the above inequality as 
\begin{align}
    \E[\|W(t+1)-\bar{W}(t+1)\|_F]
    &\leq \rootl^{t+1}\phi(t:0) \E[\|W(0)-\bar{W}(0)\|_F] + \sum_{s=1}^t \sqrt{\lambda}^{t-s+1}\phi(t:s) \frac{L}{\sqrt{s}} \label{ineq:recursive_expansion_vwt}
\end{align}
To obtain the bound on accumulation of the errors, using \eqref{ineq:recursive_expansion_vwt} we get 
\begin{align}
    &\frac{1}{\sqrt{n}}\sum_{k=0}^{t} \alpha(k+1)\sum_{i=1}^n \E[\|\vecw_i(k+1)-\bar{\vec{w}}_i(k+1)\|] \leq  \sum_{k=0}^{t} \alpha(k+1)\|W(k+1)-\bar{W}(k+1)\|_F \cr
    &\leq  \sum_{k=0}^{t} \frac{1}{\sqrt{k+1}}\rootl^{k+1} \phi(k:0)\E[\|X(0)-\bar{X}(0)\|_F] +L \sum_{k=1}^{t} \frac{1}{\sqrt{k+1}} \sum_{s=1}^k \frac{\rootl^{k+1-s} \phi(k:s)}{\sqrt{s}} \cr
    &=  c_{1}(t)\E[\|X(0)-\bar{X}(0)\|_F] +L c_2(t), \nonumber
\end{align}
where $c_1(t), c_2(t)$ are given by 
\begin{align}\label{eq:c1_c2_def}
    c_1(t):= \sum_{k=0}^{t} \frac{\rootl^{k+1}}{\sqrt{k+1}} \phi(k:0),
    \qquad c_2(t):=\sum_{k=1}^{t} \frac{1}{\sqrt{k+1}} \sum_{s=1}^k \frac{\rootl^{k+1-s} \phi(k:s)}{\sqrt{s}}.
\end{align} Using the decreasing property of $\alpha(t)$, the fact that $\phi(t)\leq 1$ for all $t\geq 0$, and the expression for a sum of a geometric series, we can uniformly bound $c_1(t)$ by $\frac{\sqrt{\lambda}}{1-\sqrt{\lambda}}$. For $c_2(t)$, note that 
\begin{align}
    c_2(t) &\leq  \sum_{k=1}^{t}  \sum_{s=1}^k \frac{\rootl^{k+1-s} \phi(k:s)}{s} 
            \leq  \sum_{k=1}^{t}  \sum_{s=1}^k \frac{\rootl^{k+1-s} }{s} \cr 
            &=  \sum_{s=1}^{t}  \frac{1}{s}\sum_{k=1}^t \rootl^{k+1-s}  \leq  \frac{\rootl}{1-\rootl}\sum_{s=1}^{t}  \frac{1}{s} \leq \frac{\rootl}{1-\rootl}(1+\ln t),  \label{ineq:bound_c2}
\end{align}
where the second inequality in \eqref{ineq:bound_c2} follows from 
\[  \sum_{s=1}^t \frac{1}{s} = 1+ \sum_{s=2}^t \frac{1}{s} \leq 1+\int_{1}^t \frac{du}{u}= 1+ \ln t.\]
Define $K_1 := \frac{\rootl}{1-\rootl}$ and $K_2:= \frac{\rootl}{1-\rootl}.$
Therefore, we have 
\begin{align}
     \frac{1}{\sqrt{n}}\sum_{k=0}^{t} \alpha(k+1)\sum_{i=1}^n\E[\|\vecw_i(k+1)-\bar{\vec{w}}_i(k+1)\|] \leq K_1 \E[\|X(0)-\bar{X}(0)\|_F] + L  K_2 (1+\ln t). \nonumber
\end{align}
Finally using the fact that $\sum_{k=0}^{t} \alpha(k+1) \geq \int_{0}^{t+1} \frac{du}{u+1} \geq  \sqrt{t+1}$ we get inequality \eqref{ineq:lemma_recursive2}.
\end{proof}
Using the accumulation of variance of the state estimates we establish an upper bound on the expected deviation of the global function at the time-averaged version of the average state estimates from the optimal value in the following lemma. 
\begin{lemma}\label{lemma:converge_f_xbar}
Under the assumptions of \Cref{thm:main1} with $\alpha(t) = 1/\sqrt{t}$ for all $t\geq 1$ and for any $\vecw^* \in \mathcal{W}^*$ we have 
\begin{align}
    \E\left[F\left( \frac{ \sum_{k=0}^t \alpha(k+1)\bar{\vecx}(k) }{ \sum_{k=0}^t \alpha(k+1)}\right) - F(\vecw^*)\right]
    &\leq  \frac{n}{2} \frac{ \|\bar{\vecx}(0)-\vecw^*\| }{\sqrt{t+1}} + \frac{ L^2 (1+\ln(t+1)) }{2n \sqrt{t+1}}\cr 
    &+ \frac{2L\sqrt{n}K_1 }{\sqrt{t+1}} \E[\|X(0)-\bar{X}(0)\|_F] + 2L^2K_2\sqrt{n}\frac{1+\ln t}{\sqrt{t+1}}, \nonumber
\end{align}
where $K_1 = K_2 = \frac{\rootl}{1-\rootl}.$
\end{lemma}
\begin{proof}
By taking expectation on both sides for the inequality  \cref{lemma:iterate_decomp}, for any $\vec{v} \in \R^d$ and $t \geq 0$ we have 
\begin{align}
    \sum_{k=0}^t \frac{2\alpha(k+1)}{n} \E[F(\bar{\vecx}(k))- F(\vec{v})] \leq \|\bar{\vecx}(0) - \vec{v} \|^2 
    &+ \sum_{k=0}^t \frac{4\alpha(k+1)}{n}\sum_{i=1}^n L_i \E[\|\vecw_i(k+1) - \bar{\vecw}(t+1) \|] \cr
    &+ \sum_{k=0}^t \alpha^2(k+1)\frac{L^2}{n^2}, \nonumber
\end{align}
since $\bar{\vecw}(t+1) = \bar{\vecx}(t)$ for all $t \geq 0$. 
Define $S(t+1) = \sum_{k=0}^t \alpha(k+1)$. Dividing the inequality above by $\frac{2 S(t+1)}{n} $ we get
\begin{align}
    \sum_{k=0}^t \frac{\alpha(k+1)}{S(t+1)} \E[F(\bar{\vecx}(k))- F(\vec{v})] \leq \frac{n}{2}\frac{\|\bar{\vecx}(0) - \vec{v} \|^2}{S(t+1)}
    &+ \sum_{k=0}^t \frac{2\alpha(k+1)}{S(t+1)}\sum_{i=1}^n L_i \E[\|\vecw_i(k+1) - \bar{\vecw}(t+1) \|] \cr
    &+ \frac{1}{S(t+1)}\sum_{k=0}^t \alpha^2(k+1)\frac{L^2}{2n}. \nonumber
\end{align}

From \cref{lemma:convergence_rate_x} we have 
\begin{align}
    &\sum_{i=1}^n \sum_{k=0}^t \frac{\alpha(k+1)}{S(t+1)}\sum_{i=1}^n L_i \E[\|\vecw_i(k+1) - \bar{\vecw}(t+1) \|] \cr 
    &\leq \frac{K_1\sqrt{n}}{\sqrt{t+1}}\E[\|X(0)-\bar{X}(0)\|_F] + LK_2\sqrt{n}\frac{(1+\ln t)}{\sqrt{t+1}}.
\end{align}
Furthermore as $\sum_{k=0}^t \alpha^2(k+1) = \sum_{k=0}^t \frac{1}{k+1} \leq 1 + \ln (t+1)$ and $S(t+1) \geq \sqrt{t+1}$, for $\vec{v} = \vecw^*$ we have 
\begin{align}
    \sum_{k=0}^t \frac{\alpha(k+1)}{S(t+1)} \E[F(\bar{\vecx}(k))- F(\vec{w}^*)] &\leq \frac{n}{2}\frac{\|\bar{\vecx}(0) - \vec{w}^* \|^2}{\sqrt{t+1}}\cr
    &+ 2 \frac{LK_1\sqrt{n}}{ \sqrt{t+1}}\E[\|X(0)-\bar{X}(0)\|_F] + L^2K_2\sqrt{n}\frac{1+\ln t}{\sqrt{t+1}}\cr
    &+ \frac{1+\ln(t+1)}{\sqrt{t+1}}\frac{L^2}{2n}\nonumber
\end{align}
which upon rearrangement gives us the result. 
\end{proof}
Finally, we provide a bound on the expected deviation of the global function computed at the time averaged version of the state estimates of any agent from the optimal value in the following theorem.
\begin{theorem}\label{thm:convergence_rate_f}
Consider the assumptions of \Cref{thm:main1} with $\alpha(t) = 1/\sqrt{t}$ for all $t\geq 1$ and   $\vecw^* \in \mathcal{W}^*$. For $\tilde \vecw_i(t+1) = \frac{\sum_{k=0}^t \alpha(k+1)\vecw_i(k+1)}{\sum_{k=0}^t\alpha(k+1)}$,  we have 
\begin{align}
      \E[F(\tilde \vecw_i(t+1)) - F(\vecw^*)]&\leq \frac{n}{2} \frac{ \|\bar{\vecx}(0)-\vecw^*\| }{\sqrt{t+1}} + \frac{ L^2 (1+\ln(t+1)) }{2n \sqrt{t+1}}\cr  
     &+ \frac{L(2\sqrt{n}+1) K_1}{\sqrt{t+1}} \E[\|X(0)-\bar{X}(0)\|_F] + L^2 K_2(2\sqrt{n}+1)\frac{1+\ln t}{\sqrt{t+1}},\nonumber
\end{align}
where $K_1 = K_2 = \frac{\rootl}{1-\rootl}.$
\end{theorem}
\begin{proof}
By the boundedness assumption of the subgradients we have 
\begin{align}
    \E[F(\tilde{\vecw}_i(t+1)) - F\left(\frac{\sum_{k=0}^t \alpha(k+1) \bar{\vecx}(k)}{\sum_{k=0}^t \alpha(k+1)}\right) &\leq \frac{L}{\sum_{k=0}^t \alpha(k+1)} \sum_{k=0}^t \alpha(k+1) \E[\|\vecw_i(t+1) - \bar{\vecx}(k)\|] \cr 
    &\leq \frac{L}{\sqrt{t+1}} (K_1\E[\|X(0)-\bar{X}(0)\|_F] + LK_2(1+\ln t))
    \nonumber
\end{align}
Using the above inequality and \Cref{lemma:converge_f_xbar} we get 
\begin{align}
    \E[F(\tilde \vecw_i(t+1)) - F(\vecw^*)] &\leq \frac{L}{\sqrt{t+1}} (K_1\E[\|X(0)-\bar{X}(0)\|_F] + LK_2(1+\ln t)) +\cr 
    & +\frac{n}{2} \frac{ \|\bar{\vecx}(0)-\vecw^*\| }{\sqrt{t+1}} + \frac{ L^2 (1+\ln(t+1)) }{2n \sqrt{t+1}}\cr 
    &+ \frac{2LK_1\sqrt{n}}{ \sqrt{t+1}} \E[\|X(0)-\bar{X}(0)\|_F] + 2L^2K_2n\frac{1+\ln t}{\sqrt{t+1}}. \cr 
    &= \frac{n}{2} \frac{ \|\bar{\vecx}(0)-\vecw^*\| }{\sqrt{t+1}} + \frac{ L^2 (1+\ln(t+1)) }{2n \sqrt{t+1}}\cr  
     &+ \frac{LK_1(2\sqrt{n}+1) }{\sqrt{t+1}} \E[\|X(0)-\bar{X}(0)\|_F] + {L^2K_2(2\sqrt{n}+1)}\frac{1+\ln t}{\sqrt{t+1}}. \nonumber
\end{align}
\end{proof}

\subsection{Discussion}
The subgradient method converges to the optimal at the rate of $O(\frac{\ln t}{\sqrt{t}})$. For randomized gossip, the convergence rate is comparable to the result that can be obtained from the result in \cref{thm:convergence_rate_f} from the result  in \cite[Theorem 2]{Nedic_2013}. However the approach in \cite{Nedic_2013} cannot be directly used to state the result in \cref{thm:convergence_rate_f} since the proof involves establishing inequality for every coordinate of the vector estimates and summing up the resulting inequalities. Such an approach cannot be extended to state-dependent averaging algorithms discussed in this work since the averaging step depends on the $\ell_2$ norm of the difference between the nodes' estimates and cannot be decoupled to establish result on individual coordinates. Another reason behind using the contraction factor approach is the lack of B-connectivity result for the interaction between the agents when using state-dependent averaging.


The hidden constant terms of the convergence rate, $O(\frac{\ln t}{\sqrt{t}})$, are influenced by the consensus algorithm used with the subgradient descent. In \Cref{thm:convergence_rate_f} the consensus step of the algorithms influences the convergence rate through the constants $K_1, K_2$ such that the convergence becomes faster as the constants decrease. Note that $K_1, K_2$ are upper bounds for $c_1(t), c_2(t)$ defined through \eqref{eq:c1_c2_def}. Based on  \cref{thm:maxEdgeContraction}, the contraction factor $\lambda =1- \frac{2\delta}{(n-1)\diam^2}$ is obtained in the following corollary, where $\delta$ for Randomized Gossip, Local Max-Gossip, Max-Gossip, and Load Balancing are provided through  \cref{prop:LMLBContracting}. 
\begin{corollary}\label{cor:constantsRate}
In \cref{thm:convergence_rate_f} the constants $K_1, K_2$ are given by $\frac{\rootl}{1-\rootl}$ which are bounded above by $n^2(n-1)\diam^2$ for Randomized Gossip, $n(n-1)\diam^2$ for Local Max-Gossip, $2(n-1)\diam^2 $ for Max-Gossip, and $(n-1)^3\diam^2 $ for Load Balancing being used as the averaging scheme with the subgradient method.
\end{corollary}
\begin{proof}
For Randomized Gossip, $1-\sqrt{\lambda} \geq \frac{1}{2}(1-\lambda) \geq \frac{1}{n^2(n-1)\diam^2}\geq \frac{1}{n^2(n-1)\diam^2}$. Therefore $K_1,K_2$ are bounded as $\frac{\rootl}{1-\rootl}\leq n^2(n-1)\diam^2.$ 

Similarly, for Local Max-Gossip we have $1-\rootl \geq \frac{1}{n(n-1)\diam^2}$ leading to $\frac{\rootl}{1-\rootl}\leq n(n-1)\diam^2$, for Max-Gossip we have $1-\rootl \geq \frac{1}{2(n-1)\diam^2}$ leading to $\frac{\rootl}{1-\rootl}\leq 2(n-1)\diam^2$, and for Load Balancing $1-\rootl \geq \frac{1}{(n-1)^3\diam^2}$ resulting in $\frac{\rootl}{1-\rootl}\leq (n-1)^3\diam^2$.
\end{proof}

\begin{remark}
We may also comment that the above result uses a conservative bound on the contraction factor $\lambda>0$. The values mentioned in \cref{cor:constantsRate} are upper bounds on the constants in the convergence rate. However, tighter bounds on the constants $K_1, K_2$ are possible. 
For Randomized Gossip, the contraction factor can be improved to the square of the second largest eigenvalue of the expected averaging matrix $\E[A(t,X(t))]$.  

In principle, in the proof of Theorem~5, for each of the state-dependent algorithm, such a contraction factor would depend on the sample path (past trajectory) of the dynamics. For example, when the consensus scheme used is Load Balancing, we know that for most practical purposes, when the nodes do not have multiple neighbors with maximal disagreement, the constant $\delta$ in \cref{prop:LMLBContracting} is even grater than $\frac{1}{2}$, more precisely, it is $\frac{C_e(X)}{2}$, where $C_e$ is the number of edges over which the exchange is taking place in the averaging step with the state estimate $X\in \R^{n\times d}$. With the improved $\delta$, the bound on the constants $K_1,K_2$ can be improved to $\frac{(n-1)\diam^2}{2C_e(X(t))}\leq \frac{(n-1)\diam^2}{2}$.

Similarly the bounds on the convergence rate for Local Max-Gossip can be improved by using tighter contraction factor for the averaging matrices. However as seen from \cite[Theorem 2]{ustebay2010greedy}, the contraction factor may take cumbersome form which cannot be readily used to establish better bounds on $c_1(t),c_2(t)$.

The problem of finding useful convergence rate for state-dependent averaging is a non-trivial open problem. 
\end{remark}
\color{black}